\PassOptionsToPackage{dvipsnames}{xcolor}

\documentclass[10pt]{article}

\usepackage{amsfonts,amssymb,amsmath,amsopn,float}
\usepackage{amsthm}
\usepackage{comment}
\usepackage{xifthen}
\usepackage{graphicx}
\usepackage{tikz}
\usepackage{pgfplots}
\usepgfplotslibrary{units}
\usepackage{epstopdf}
\usepackage[utf8]{inputenc}
\usepackage{mathtools}
\usepackage{url}
\usepackage{subcaption}
     \makeatletter
\def\@fnsymbol#1{\ensuremath{\ifcase#1\or \dagger\or *\or \ddagger\or
     \mathsection\or \mathparagraph\or \|\or **\or \dagger\dagger
     \or \ddager\ddager \else\@cterr\fi}}
     \makeatother
\newcommand{\Ltwo}[1]{%
	\ifthenelse{\equal{#1}{}}{L^2}{L^2(#1)}%
}

\newcommand{\Ltwoz}[1]{%
	\ifthenelse{\equal{#1}{}}{L^2_0}{L^2_0(#1)}%
}

\newcommand{\Cone}[1]{%
	\ifthenelse{\equal{#1}{}}{C^{1}}{C^{1}(#1)}%
}

\newcommand{\Conez}[1]{%
	\ifthenelse{\equal{#1}{}}{C^{1}_{0}}{C^{1}_{0}(#1)}%
}

\newcommand{\Ctwo}[1]{%
	\ifthenelse{\equal{#1}{}}{C^{2}}{C^2(#1)}%
}

\newcommand{\Ctwoz}[1]{%
	\ifthenelse{\equal{#1}{}}{C^{2}_{0}}{C^{2}_{0}(#1)}%
}

\newcommand{\Cholder}[1]{%
	\ifthenelse{\equal{#1}{}}{C^{0,\gamma}}{C^{0,\gamma}(#1)}%
}

\newcommand{\Cholderz}[1]{%
	\ifthenelse{\equal{#1}{}}{C^{0,\gamma}_{0}}{C^{0,\gamma}_{0}(#1)}%
}

\newcommand{\bolds}[1]{\boldsymbol{#1}}

\newcommand{\bb}{\bolds{b}}
\newcommand{\bc}{\bolds{c}}

\newcommand{\be}{\bolds{e}}
\newcommand{\bff}{\bolds{f}}
\newcommand{\bg}{\bolds{g}}

\newcommand{\bn}{\bolds{n}}

\newcommand{\bu}{\bolds{u}}
\newcommand{\bv}{\bolds{v}}
\newcommand{\bw}{\bolds{w}}
\newcommand{\bx}{\bolds{x}}
\newcommand{\by}{\bolds{y}}

\newcommand{\nb}[3]{%
  {\colorbox{#2}{\bfseries\scriptsize\textcolor{white}{#1}}}%
  {\textcolor{#2}{\textit{#3}}}}
\newcommand{\rl}[1]{\nb{RPL}{BlueViolet}{#1}}

\newcommand{\sautoref}[2]{\hyperref[#2]{#1 \ref*{#2}}}

\newtheorem{theorem}{Theorem}[section]
\newtheorem{lemma}{Lemma}[section]
\newtheorem{hypothesis}{Hypothesis}[section]
\newtheorem{definition}{Definition}[section]
\numberwithin{equation}{section}

\title{ \bf  Nonlocal elastodynamics and fracture\thanks{This material is based upon work supported by the U. S. Army Research Laboratory and the U. S. Army Research Office under contract/grant number W911NF1610456.}}

\author{Robert P. Lipton\thanks{Department of Mathematics, Louisiana State University,
              Baton Rouge, LA 70803,
              Orcid: https://orcid.org/0000-0002-1382-3204,
              \tt{$^*$lipton@lsu.edu}
}
\and
Prashant K. Jha\thanks{Oden Institute for Computational Engineering and Sciences,
	      The University of Texas at Austin,
	      Austin, TX 78712,
	      Orcid: https://orcid.org/0000-0003-2158-364X, 
	      \tt{ pjha@utexas.edu}}
}

\begin{document}

\maketitle
\date{}

\begin{abstract}
A nonlocal field theory of peridynamic type is applied to model the brittle fracture problem.
The elastic fields obtained from the nonlocal model are shown to converge in the limit of vanishing non-locality to solutions of classic plane elastodynamics associated with a running crack. 
\end{abstract}
\begin{flushleft}
\end{flushleft}
   {\bf AMS Subject} 34A99, 74R99
\section{Introduction}

Fracture can be viewed as a collective interaction across large and small length scales. With the application of enough stress or strain to a brittle material, atomistic scale bonds will break, leading to fracture of the macroscopic specimen.  From a modeling perspective fracture should appear as an emergent phenomena generated by an underlying field theory eliminating the need for a supplemental kinetic relation describing crack growth. The displacement field inside the body for points $\bx$  at time $t$ is written $\bu(\bx,t)$.
The perydynamic model  \cite{CMPer-Silling}, \cite{States}, 
is described by the nonlocal balance of linear momentum of the form
\begin{equation}\label{eqn-uf}
\begin{aligned}
   \rho{{\bu}_{tt}}(\bx,t) = \int_{{\mathcal H}_{\epsilon}({\bx})} {\bolds{f}}(\by,\bx)\;d\by + \bb(\bx,t)
\end{aligned}
\end{equation}
where $\mathcal{H}_{\epsilon}({\bx})$ is a neighborhood of $\bx$,  $\rho$ 
is the density, $\bb$ is the body force 
density field, and $\bolds{f}$ is a material-dependent constitutive law
that represents the force density that a point $\by$ inside the neighborhood exerts on $\bx$ as a result of the deformation field.
The radius $\epsilon$ of the neighborhood is referred to as the \emph{horizon}. 
Here all points satisfy the same field equation \eqref{eqn-uf}. 
The displacement fields and fracture evolution predicted by the nonlocal model should agree with the dynamic fracture of specimens when the length scale of non-locality is sufficiently small. In this respect numerical simulations are compelling, see for example \cite{Bobaru 2015}, \cite{silling05}, and \cite{ParksTrasketal}.

The displacement for the nonlocal theory is examined in the limit of vanishing non-locality. This is done for a class of peridynamic models with nonlocal forces derived from double well potentials see,  \cite{CMPer-Lipton}.  The term double well describes the force potential between two points. One of the wells is degenerate and appears at infinity while the other is at zero strain.  For small strains the nonlocal force is linearly elastic but for larger strains the  force begins to soften and then approaches zero after reaching a critical strain. This type of nonlocal model is called a cohesive model. Fracture energies of this type have been defined for displacement gradients in \cite{trusk} with the goal of understanding fracture as a phase transition in the framework of \cite{erick}.

We theoretically investigate the limit of the displacements for the cohesive model as the length scale $\epsilon$ of nonlocal interaction goes to zero.  All information on this limit is obtained from what is known from the nonlocal model for $\epsilon>0$. In this paper the single edge notch specimen is considered as given in figure \ref{singlenotch} and the target theory governing the evolution of  displacement fields is identified when $\epsilon=0$.  
\begin{figure} 
\centering
\begin{tikzpicture}[xscale=0.60,yscale=0.60]

%
%
\draw [-,thick] (-1.80,0.05) -- (-1.3,0.05);

\draw [-,thick] (-1.80,-0.05) -- (-1.3,-0.05);


\draw [-,thick] (-1.3,0.05) to [out=0, in=90] (-1.25 ,0);

\draw [-,thick] (-1.25,0) to [out=-90, in=0] (-1.3 ,-0.05);


\draw [-,thick] (-2,0.2) to [out=-90, in=180] (-1.8 ,0.05);

\draw [-,thick] (-2,-0.2) to [out=90, in=180] (-1.8 ,-0.05);

\draw [-,thick] (-2,2.8) to [out=90, in=180] (-1.8 ,3.0);

\draw [-,thick] (-2,-2.8) to [out=-90, in=180] (-1.8 ,-3.0);

\draw [-,thick] (1.8,3.0) to [out=0, in=90] (2.0 ,2.8);

\draw [-,thick] (1.8,-3.0) to [out=0, in=-90] (2.0 ,-2.8);


%
%
\draw [-,thick] (-2,2.8) -- (-2,0.2);
\draw [-,thick] (-2,-2.8) -- (-2,-0.2);
\draw [-,thick] (-1.8,-3) -- (1.8,-3);
\draw [-,thick] (2,2.8) -- (2,-2.8);
\draw [-,thick] (1.8,3) -- (-1.8,3);
%
%
%
%
\node [right] at (-0.5,0.0) {{\Large $D$}};

%
%

\draw[->,thick] (-1.8,3.0) -- (-1.8,3.80);

\draw[->,thick] (-0.9,3.0) -- (-0.9,3.80);

\draw[->,thick] (0.0,3.0) -- (0.0,3.80);

\draw[->,thick] (0.9,3.0) -- (0.9,3.80);

\draw[->,thick] (1.8,3.0) -- (1.8,3.80);


\draw[->,thick] (-1.8,-3.0) -- (-1.8,-3.80);

\draw[->,thick] (-0.9,-3.0) -- (-0.9,-3.80);

\draw[->,thick] (0.0,-3.0) -- (0.0,-3.80);

\draw[->,thick] (0.9,-3.0) -- (0.9,-3.80);

\draw[->,thick] (1.8,-3.0) -- (1.8,-3.80);

\end{tikzpicture} 
\caption{{ \bf Single-edge-notch}}
 \label{singlenotch}
\end{figure}
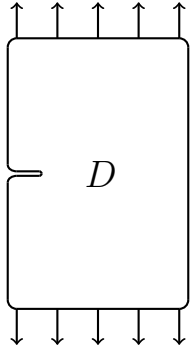
One of the hallmarks of peridynamic simulations is localization of defect sets with horizon as $\epsilon\rightarrow 0$. Theoretically localization of the jump set of the displacement is established as $\epsilon\rightarrow 0$ in \cite{CMPer-Lipton3},  \cite{CMPer-Lipton} where the limiting displacement is shown to be an $SBD^2(D)$ valued function for almost all times $t\in [0,T]$, see section \ref{ss:crackwaveinteraction}. The nonlocal cohesive  model converges to a dynamic model having bounded Griffith fracture energy associated with brittle fracture and elastic displacement fields satisfying the elastic wave equation \cite{CMPer-Lipton3},  \cite{CMPer-Lipton} away from the fractures. This can be seen for arbitrarily shaped specimens with smooth boundary in two and three dimensions. However the explicit traction law relating the crack boundary to the elastic field lies out side the scope of that analysis.

This paper builds on earlier work and provides a  global description of the limit dynamics describing elastic fields surrounding a crack for the single edge notch pulled apart by traction forces on its top and bottom edges. The objective of this paper is to show that the elastic fields seen in the  nonlocal model are consistent with those in the local model in the limit of vanishing horizon. The analysis given here shows that it is possible to recover the boundary value problem for the linear elastic displacement given by Linear Elastic Fracture Mechanics inside a cracking body as the limit of a nonlocal fracture model.  To illustrate this a family of initial value problems given in the nonlocal formulation is prescribed. The family is parameterized by horizon size $\epsilon$. The crack motion for $\epsilon>0$ is prescribed by the solutions of the nonlocal initial value problem. It is shown that up to subsequences that as $\epsilon\rightarrow 0$ the displacements associated with the solution of the nonlocal model converge in mean square uniformly in time to the limit 
displacement $\bu^0(\bx,t)$ that satisfies:
\begin{itemize}

\item  Prescribed inhomogeneous traction boundary conditions.

\item Balance of linear momentum as described by the linear elastic wave equation off the crack.

\item Zero traction on the sides of the evolving crack.

\item The set on which the elastic displacement jumps is a subset of the crack set.

\item The limiting crack motion is determined by the sequence of nonlocal problems for $\epsilon>0$ and is obtained in the $\epsilon=0$ limit. 

\end{itemize}
The first four items deliver the boundary conditions, elastodynamic equations, traction boundary conditions on the crack, and correlation between displacement jumps and crack set articulated in the theory of dynamic Linear Elastic Fracture Mechanics (LEFM)  \cite{Freund}, \cite{RaviChandar}, \cite{Anderson}, 
\cite{Slepian}.  The $\epsilon\rightarrow 0$ limit of displacement fields for the nonlocal model is  seen to be a weak solution for the wave equation on a time dependent domain recently defined in the work of \cite{DalToader}, see theorem \ref{zerohorizonweaksoln}. Here the time dependent domain is given by the domain surrounding the moving crack. This establishes a rigorous connection between the nonlocal fracture formulation using a peridynamic model derived from a double well potential and the wave equation posed on cracking domains given in \cite{DalToader}.  If one assumes a more general crack structure for the nonlocal problem then a connection to the local problem for more general time dependent domains can be made, this is discussed in the conclusion. 

The analysis treats a  dynamic problem and
compactness methods suited to the balance of momentum for nonlocal - nonlinear operators, are applied, see lemma \ref{convergences} and theorem \ref{limitspace} . Proceeding this way delivers the zero traction condition on the crack lips for the fracture model in the local limit. 
Another issue is to prescribe body forces for the nonlocal model that transform to into boundary tractions for the local model. In this paper a suitable layer of force is prescribed adjacent to  the boundary of the sample for the nonlocal model. It is motivated by the one proposed in \cite{ParksTrasketal}. The layer of force is shown to converge to the standard traction boundary conditions seen in local models, see lemma \ref{convergingrhs}. This theoretically corroborates the numerical experiments with the nonlocal model carried out in \cite{ParksTrasketal}. It is pointed out that the nonlocal model considered here is elastic, so cracks can heal if the strain across the crack drops below the critical value. However, in this paper the initial conditions and boundary conditions are chosen such that the specimen is under tensile strain and pulled apart so the crack has no opportunity to heal. More complex models \cite{CMPer-Lipton5} involving dissipation and non-monotone or cyclic load paths lie outside the scope of the paper and provide interesting avenues for future research.

The nonlocal model is an example of several new approaches to dynamic fracture modeling.  These include solution of the wave equation on cracking domains  \cite{DalLarsenToader}, \cite{DalLar}, \cite{DalToader}, \cite{Nicaise}, 
phase field methods, \cite{Hughes},  \cite{BourdinLarsenRichardson}, \cite{Miehe}, \cite{Ortiz},  and peridynamics
\cite{CMPer-Silling}, \cite{States}, \cite{Bobaru 2015}, \cite{ParksTrasketal}. In the absence of fracture and dynamics the $\Gamma$ convergence approach has been applied to peridynamic boundary value problems. The nonlocal formulations are shown to converge to equilibrium boundary value problems for hyperelastic and elastic materials as $\epsilon\rightarrow 0$, see \cite{Bellido}, \cite{CMPer-Mengesha2}. It is noted that the aforementioned references while relevant to this work are only a few from a rapidly expanding literature.

The paper is organized as follows: In section \ref{s:nonlocal dynamics} the nonlocal constitutive law as derived from a double well potential is described and the nonlocal boundary value problem describing crack evolution is given. 
Section \ref{ss:crackwaveinteraction} provides the principle results of the paper and describes the convergence of the displacement fields given in the nonlocal model to  the elastic displacement field satisfying,  the linear wave equation off the crack set, zero Neumann conditions on the crack, and traction boundary conditions. Existence and uniqueness for the nonlocal problems are established in \ref{existenceuniqueness}. The convergence theorems are proved in sections \ref{s:proofofsymmetry} and \ref{s:proofofconvergence}. The proof that the limit displacement is a weak solution of the wave equation on a time dependent domain is given in section \ref{s:weaksolnproof}. The results are summarized in the conclusion section \ref{s:conclusions}.

\section{Nonlocal Elastodynamics}\label{s:nonlocal dynamics}

In this section we formulate the nonlocal dynamics as an initial boundary value problem driven by a layer of force adjacent to the boundary.  Here all quantities are non-dimensional.  Define the region $D$ given by a notched rectangle with rounded corners, see figure \ref{singlenotch}.  The domain lies within the rectangle $\{0<x_1<a;\,-b/2<x_2<b/2\}$ and the notch originates on the left side of the specimen and is of thickness $2d$ and total length $\ell(0)$ with a circular tip and rounded corners, see figure \ref{singlenotch}. 
The domain is subject to plane strain loading and we will assume small deformations so the deformed configuration is the same as the reference configuration.  We have $\bu=\bu(\bx,t)$ as a function of space and time but will  suppress the $\bx$ dependence when convenient and write $\bu(t)$. The tensile strain $S$ between two points $\bx,\by$ in $D$ along the direction $\be_{\by-\bx}$ is defined as
\begin{align}\label{strain}
S(\by,\bx,\bu(t))=\frac{\bu(\by,t)-\bu(\bx,t)}{|\by-\bx|}\cdot \be_{\by-\bx},
\end{align}
where $ \be_{\by-\bx}=\frac{\by-\bx}{|\by-\bx|}$ is a unit vector and ``$\cdot$'' is the dot product.

The nonlocal force ${\bolds{f}}$ is defined in terms of a double well potential that is a function of the strain $S(\by,\bx,\bu(t))$.
The force potential is defined for all $\bx,\by$ in $D$ by
\begin{equation}\label{tensilepot}
\mathcal{W}^\epsilon (S(\by,\bx,\bu(t)))=J^\epsilon(|\by-\bx|)\frac{1}{\epsilon^3\omega_2|\by-\bx|}\Psi(\sqrt{|\by-\bx|}S(\by,\bx,\bu(t)))
\end{equation}
where $\mathcal{W}^\epsilon (S(\by,\bx,\bu(t)))$ is the pairwise force potential per unit length between two points $\bx$ and $\by$. 
Here, the influence function $J^\epsilon(|\by-\bx|)$ is a measure of the influence that the point $\by$ has on $\bx$. Only points inside the horizon can influence $\bx$ so $J^\epsilon(|\by-\bx|)$ is nonzero for $|\by - \bx| < \epsilon$ and is zero otherwise. We take $J^\epsilon$ to be of the form: $J^\epsilon(|\by - \bx|) = J(\frac{|\by - \bx|}{\epsilon})$ with $J(r) = 0$ for $r\geq 1$ and $0\leq J(r)\leq M < \infty$ for $r<1$.

The force potential is described in terms of its potential function and to fix ideas
 $\Psi$ is given by 
\begin{equation}\label{potential}
 \Psi=h(r^2) 
 \end{equation}
 where $h$ is concave, see figure \ref{ConvexConcave}(a). Here $\omega_2$ is the area of the unit disk and $\epsilon^2\omega_2$ is the area of the horizon $\mathcal{H}_{\epsilon}(\bx)$.
The potential function $\Psi$ represents a convex-concave potential such that the associated force acting between material points $\bx$ and $\by$ are initially elastic and then soften and decay to zero as the strain between points increases, see  figure \ref{ConvexConcave}(b). The force between $\bx$ and $\by$ is referred to as the bond force. The first well for $\mathcal{W}^\epsilon (S(\by,\bx,\bu(t)))$ is at zero tensile strain and the potential function satisfies
\begin{align}
\label{choice at oregon}
\Psi(0)=\Psi'(0) = 0.
\end{align}
The well for $\mathcal{W}^\epsilon (S(\by,\bx,\bu(t)))$ in the neighborhood of infinity is characterized by the horizontal asymptote $\lim_{S\rightarrow \infty} \Psi(S)=C^+$, see figure \ref{ConvexConcave}(a). The critical tensile strain $S_c>0$ for which the force begins to soften is given by the inflection point ${r}^c>0$ of $g$ and is 
\begin{equation}
S_c=\frac{{r}^c}{\sqrt{|\by-\bx|}},
\label{crittensileplus}
\end{equation}
and $S_+$ is the strain at which the force goes to zero 
\begin{equation}
S_+=\frac{{r}^+}{\sqrt{|\by-\bx|}}.
\label{crittensileplus2}
\end{equation}
We assume here that the  potential functions are bounded and are smooth. It is pointed out that for this modeling the bond force in compression allows for eventual softening. However one can easily generalize the analysis to handle an asymmetric bond force that resists compression.

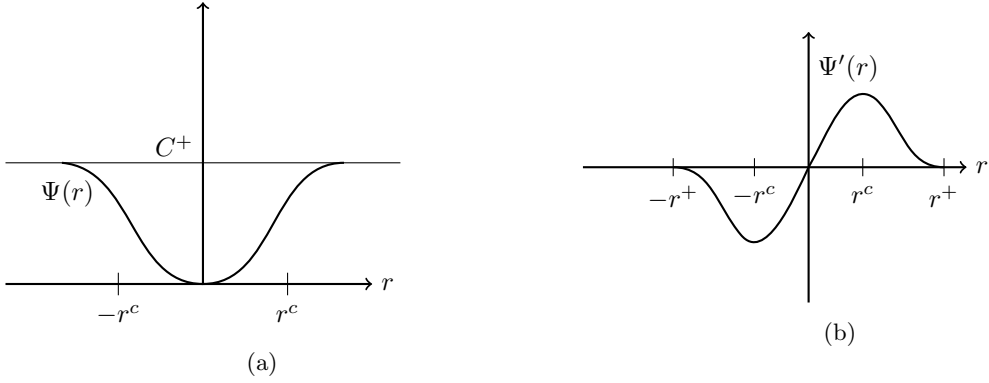
\begin{figure}
    \centering
    \begin{subfigure}{.45\linewidth}
        \begin{tikzpicture}[xscale=0.75,yscale=0.75]
		    \draw [<->,thick] (0,5) -- (0,0) -- (3.0,0);
			\draw [-,thick] (0,0) -- (-3.5,0);
			\draw [-,thin] (0,2.15) -- (3.5,2.15);
			\draw [-, thin] (-3.5,2.15) -- (0,2.15);
			\draw [-,thick] (0,0) to [out=0,in=-180] (2.5,2.15);
			\draw [-,thick] (-2.5,2.15) to [out=-5,in=180] (0,0);
			
			\draw (1.5,-0.2) -- (1.5, 0.2);
			\node [below] at (1.5,-0.2) {${r}^c$};
			
			\draw (-1.5,-0.2) -- (-1.5, 0.2);
			\node [below] at (-1.5,-0.2) {${-r}^c$};

			\node [right] at (3,0) {$r$};
			\node [left] at (0,2.45) {$C^+$};
			\node [above] at (-2.4,1.20) {$\Psi(r)$};
		  \end{tikzpicture}
		  \caption{}
		  \label{ConvexConcavea}
    \end{subfigure}
    \hskip2em
    \begin{subfigure}{.45\linewidth}
        \begin{tikzpicture}[xscale=0.6,yscale=0.6]
		    \draw [<-,thick] (0,3) -- (0,-3);
			\draw [->,thick] (-5,0) -- (3.5,0);
			\draw [-,thick] (0,0) to [out=60,in=140] (1.5,1.5) to [out=-45,in=180] (3,0.0);
			
			\draw [-,thick] (-3.0,-0.0) to [out=0,in=130] (-1.5,-1.5) to [out=-50, in=245] (0,0);
			
			\draw (3.0,-0.2) -- (3.0, 0.2);
			\draw (-3.0,-0.2) -- (-3.0, 0.2);
			\draw (1.2,-0.2) -- (1.2, 0.2);
			\draw (-1.2,-0.2) -- (-1.2, 0.2);
			\node [below] at (1.2,-0.2) {${r}^c$};
			\node [below] at (-1.2,-0.2) {$-{r}^c$};
			\node [below] at (3.0,-0.2) {${r}^+$};
			\node [below] at (-3.0,-0.2) {$-{r}^+$};
			\node [right] at (3.5,0) {${r}$};
			\node [right] at (0,2.2) {$\Psi'(r)$};
		  \end{tikzpicture}
		   \caption{}
		   \label{ConvexConcaveb}
    \end{subfigure}
    \caption{\bf (a) The double well potential function $\Psi(r)$  for tensile force. Here $C^+$ is the asymptotic value of $\Psi$. (b) Cohesive force. The derivative of the force potential goes smoothly to zero at $\pm r^+$.}\label{ConvexConcave}
\end{figure}

\subsection{Peridynamic equation of motion}
The potential energy of the motion is given by
\begin{equation}\label{new peri}
\begin{aligned}
PD^\epsilon(\bu)=\int_D \int_{\mathcal{H}_\epsilon(\bx)\cap D} |\by-\bx|\mathcal{W}^\epsilon(S(\by,\bx,\bu(t)))\,d\by d\bx.
\end{aligned}
\end{equation}
We consider single edge notched specimen  $D$ pulled apart by an $\epsilon$ thickness layer of body force on the top and bottom of the domain consistent with plain strain loading. In the nonlocal setting the ``traction'' is given by the layer of body force on the top and bottom of the domain. For this case the body force is written as 
\begin{equation}\label{bodyforce}
\begin{aligned}
\bb^\epsilon(\bx,t)=\be^2\epsilon^{-1}{g(x_1,t)}\chi_+^\epsilon(x_1,x_2) \hbox{  on the top layer and}\\
\bb^\epsilon(\bx,t)=-\be^2\epsilon^{-1}{g(x_1,t)}\chi_-^\epsilon(x_1,x_2)\hbox{  on the bottom  layer,}
\end{aligned}
\end{equation}
where $\be^2$ is the unit vector in the vertical direction, $\chi^\epsilon_+$ and $\chi^\epsilon_-$ are the characteristic functions of the boundary layers given by
\begin{equation}\label{layers}
\begin{aligned}
\chi_+^\epsilon(x_1,x_2)&=1\,\,\,\hbox{on $\{\theta<x_1<a-\theta, \,b/2-\epsilon<x_2<b/2\}$ and $0$ otherwise,}\\
\chi_-^\epsilon(x_1,x_2)&=1\,\,\,\hbox{on $\{\theta<x_1<a-\theta, \,-b/2<x_2<-b/2+\epsilon\}$ and $0$ otherwise},
\end{aligned}
\end{equation}
where $\theta$ is the radius of curvature of the rounded corners of $D$.
The top and bottom traction forces are equal and in opposite directions and
$g(x_1,t)>0$. We take the function $g$ to be smooth and bounded in the variables $x_1$  and $t$ and define $\bg$ on $\partial D$
such that
\begin{equation}\label{l2}
\begin{aligned}
\bg=\pm\be^2g \hbox{ on } \{\theta\leq x_1\leq a-\theta,\,\,x_2 = \pm b/2\}\text{ and } \bg=0\hbox{ elsewhere on }\partial D.
\end{aligned}
\end{equation}

For any in-plane rigid body motion $\bw(\bx)={\mathbf \Omega}\times\bx+\bc$ where ${\mathbf \Omega}$ and $\bc$ are constant vectors we see that
\begin{equation}
\label{rigid}
\begin{aligned}
\int_D\bb^\epsilon\cdot\bw\,d\bx=0  \hbox{        and        }  S(\by,\bx,\bw)=0,
\end{aligned}
\end{equation}
and we show in lemma \ref{convergingrhs} that $\bb^\epsilon$, is a bounded linear functional on an appropriate Sobolev space and converges as $\epsilon\rightarrow 0$ to a boundary traction.

For future reference we denote the space of all square integrable fields orthogonal to rigid body motions in the $L^2$ inner product  by 
\begin{equation}
\label{rigidbody}
\dot L^2(D;\mathbb{R}^2).
\end{equation}

In this treatment the density $\rho$ is assumed constant and we define the {Lagrangian}  $${\rm{L}}(\bu,\partial_t \bu,t)=\frac{\rho}{2}||\dot \bu||^2 _{L^2 (D;\mathbb{R}^2)}-PD^\epsilon(\bu)+\int_D \bb^\epsilon\cdot \bu \,d\bx,$$
where $\dot \bu=\frac{\partial \bu}{\partial t}$ is the velocity. The action integral for a time evolution over the interval $0<t<T,$  is given by
\begin{equation}\label{action}
\begin{aligned}
I=\int_0^T{\rm{L}}(\bu,\partial_t \bu,t)\,dt.
\end{aligned}
\end{equation}
We suppose $\bu^\epsilon(t)$ is a stationary point and $\bw(t)$ is a perturbation and applying the {principal of least action} gives the nonlocal dynamics
\begin{eqnarray}\label{energy based weakform}
\begin{aligned}
&\rho \int_0^T\int_D\dot{\bu}^\epsilon(\bx,t)\cdot {\dot\bw}(\bx,t)d\bx\,dt\\
 &=
\int_0^T\int_D \int_{\mathcal{H}_\epsilon(\bx)\cap D} |\by-\bx|\partial_S\mathcal{W}^\epsilon(S(\by,\bx,\bu^\epsilon(t)))S(\by,\bx,\bw(t))\,d\by d\bx\,dt\\
&-\int_0^T\int_D\bb^\epsilon(\bx,t)\cdot\bw(\bx,t)d\bx\,dt.
\end{aligned}
\end{eqnarray}
and an integration by parts gives the strong form 
\begin{equation}\label{energy based model2}
\begin{aligned}
\rho \ddot{\bu}^\epsilon(\bx,t)=\mathcal{L}^\epsilon(\bu^\epsilon)(\bx,t)+\bb^\epsilon(\bx,t),\hbox{  for  $\bx\in D$}.
\end{aligned}
\end{equation}
Here $\mathcal{L}^\epsilon(\bu^\epsilon)$ is the peridynamic force 
\begin{align}\label{pdforce}
\mathcal{L}^\epsilon(\bu^\epsilon)=\int_{\mathcal{H}_{\epsilon}(x)\cap D} {\bolds{f}}^\epsilon(\by,\bx)\;d\by
\end{align}
and ${\bolds{f}}^\epsilon(\bx,\by)$ is given by
\begin{align}\label{nonlocforcetensite}
{\bolds{f}}^\epsilon(\bx,\by)&=2\partial_S\mathcal{W}^\epsilon(S(\by,\bx,\bu^\epsilon(t)))\be_{\by-\bx},
\end{align}
where
\begin{align}\label{derivbond}
\partial_S\mathcal{W}^\epsilon(S(\by,\bx,\bu^\epsilon(t)))&=\frac{1}{\epsilon^{3} \omega_2}\frac{J^\epsilon(|\by-\bx|)}{|\by-\bx|}\partial_S \Psi(\sqrt{|\by-\bx|}S(\by,\bx,\bu^\epsilon(t))).
\end{align}

The dynamics is complemented with the initial data 
\begin{align}\label{idata}
\bu^\epsilon(\bx,0)=\bu_0(\bx), \qquad  \partial_t \bu^\epsilon(\bx,0)=\bv_0(\bx).
\end{align}
Where $\bu_0$ and $\bv_0$ lie in $\dot L^2(D;\mathbb{R}^2)$. 

The initial value problem for the nonlocal evolution  given by \eqref{energy based model2} and \eqref{idata}
or equivalently by \eqref{energy based weakform} and \eqref{idata} has a unique solution in $C^2([0,T];\dot L^2(D;\mathbb{R}^2))$, see section \ref{existenceuniqueness}. Application of Gronwall's inequality shows that the nonlocal evolution $\bu^\epsilon(\bx,t)$ is uniformly bounded in the mean square norm over the time interval $0<t<T$, 
\begin{eqnarray}
\max_{0<t<T}\left\{\Vert \bu^\epsilon(\bx,t)\Vert_{L^2(D;\mathbb{R}^2)}^2\right\}<K,
\label{bounds}
\end{eqnarray}
where the upper bound $K$ is independent of $\epsilon$ and depends only on the initial conditions and body force applied up to time $T$, see  \cite{CMPer-Lipton}.

\subsection{\bf Failure zone and softening zone geometry}
\label{failurezone}
The failure zone and softening zone are introduced and described for the $\epsilon>0$ model.
The failure zone $FZ^\epsilon(t)$  represents the crack in the nonlocal model at a given time $t$. 
This is the set of pairs $\bx$ and $\by$ with $|\by-\bx|<\epsilon$ for which the force ${\bolds{f}}^\epsilon(\bx,\by)$ acting between them is zero. In this problem the domain and body force adjacent to the upper and lower boundaries are symmetric with respect to the $x_2=0$ axis, see \eqref{bodyforce}. The body force is perpendicular to the $x_2=0$ axis and points in the $\be^2$ direction on the top boundary layer and the $-\be^2$ direction on the bottom boundary layer.  Choosing initial conditions appropriately the solution to the initial value problem has its first component $u_1^\epsilon$ even with respect to the $x_2=0$ axis and second component  $u_2^\epsilon$ odd for $t\in [0,T]$.  For the time dependent body force chosen here the failure is in tension and confined to a neighborhood of the $x_2=0$ axis of width $2\epsilon$ where strains are largest. The failure zone nucleated at the notch and the  failure zone is defined by a centerline lying on the $x_2=0$ axis. The failure zone propagates from the notch into the interior of the specimen. The tip of the notch is defined to be the point $(x_1=\ell(0),\,x_2=0)$. The failure zone centerline is 
\begin{equation}\label{center}
C^\epsilon=\{\ell(0)\leq x_1\leq\ell^\epsilon(t),\,x_2=0\}.  
\end{equation}
 The failure zone is written as
\begin{equation}
\begin{aligned}
FZ^\epsilon(t)=\{\bx\hbox{ and }\by\in D,\,|\by-\bx|<\epsilon:\,\, \bx+s(\by-\bx) \cap C^\epsilon(t)\not=\emptyset, \hbox{ for some } s\in [0,1]\},
\end{aligned}
\label{failure}
\end{equation}
The centerline is shown in figure \ref{P} and the failure zone is the shaded region.
Here ${\bolds{f}}^\epsilon(\bx,\by)=0$ for $\bx$ and $\by$ in $FZ^\epsilon(t)$.
The crack motion is prescribed by the monotonically increasing function $\ell^\epsilon(t)$ for every $\epsilon>0$, $t\in[0,T]$. We will assume that the crack does not propagate all the way through the sample, i.e., $\ell^\epsilon(T)<a-\delta$, for every $\epsilon$ where $\delta$ is a small fixed positive constant.

The total traction force on on the layer of thickness $\epsilon$ above the failure zone centerline exerted by the body below the failure zone centerline is null and vice versa.
Associated with the failure zone is the softening zone. The softening zone $SZ^\epsilon(t)$ is the set of pairs $\bx$ and $\by$ with $|
\by-\bx|<\epsilon$ separated by the $x_2=0$ axis such that the force ${\bolds{f}}^\epsilon(\bx,\by)$  between them is non-increasing with increasing strain.  
From this it is clear that $FZ^\epsilon(t)\subset SZ^\epsilon(t)$.
Furthermore at the leading edge of the crack one sees force softening between points $\bx$ and $\by$ separated by less than  $\epsilon$ on either side of the $x_2=0$ axis. As the crack centerline moves forward passing between $\bx$ and $\by$ the force between $\bx$ and $\by$ decreases to zero, see figure \ref{P}. That is given $t$ there is a later time $t+\Delta t$ for which $FZ^\epsilon(t+\Delta t)=SZ^\epsilon(t)$. The process zone where the bonds have softened but not failed, i.e., $\bx,\by \in SZ^\epsilon(t)\setminus FZ^\epsilon(t)$ is of length proportional to $\epsilon$. The softening zone $SZ^\epsilon(t)$ is specified through a softening zone centerline.
The force between two points $\bx$ and $\by$ separated by the softening zone centerline decreases with time. The centerline is
\begin{equation}
\begin{aligned}
 S^\epsilon(t)=\{\ell(0)\leq x_1\leq\ell^\epsilon(t)+C\epsilon,\, x_2=0\},
\end{aligned}
\label{sc}
\end{equation}
where $C$ is a positive constant. The softening zone is written as
\begin{equation}
\begin{aligned}
SZ^\epsilon(t)=\{\bx\hbox{ and }\by\in D,\,|\by-\bx|<\epsilon:\,\,  \bx+s(\by-\bx) \cap S^\epsilon(t)\not=\emptyset, \hbox{ for some } s\in [0,1]\} .
\end{aligned}
\label{soften}
\end{equation}

The strain $S(\by,\bx,\bu^\epsilon(t))$ is decomposed for $\bx$ and $\by$ in $D$ and $|\by-\bx|<\epsilon$ as 
\begin{eqnarray}
S(\by,\bx,\bu^\epsilon(t))=S(\by,\bx,\bu^\epsilon(t))^- +S(\by,\bx,\bu^\epsilon(t))^+
\label{decomposeS}
\end{eqnarray}
where
\begin{equation}
S(\by,\bx,\bu^\epsilon(t))^-=\left\{\begin{array}{ll}S(\by,\bx,\bu^\epsilon(t)),&\hbox{if  $|S(\by,\bx,\bu^\epsilon(t))|<S_c$}\\
0,& \hbox{otherwise}
\end{array}\right.
\label{decomposedetailsS}
\end{equation}
and
\begin{equation}
S(\by,\bx,\bu^\epsilon(t))^+=\left\{\begin{array}{ll}S(\by,\bx,\bu^\epsilon(t)),&\hbox{if  $|S(\by,\bx,\bu^\epsilon(t))|\geq S_c$}\\
0,& \hbox{otherwise}
\end{array}\right.
\label{decomposedetailsS2}
\end{equation}
with
\begin{equation}
\begin{aligned}
&S(\by,\bx,\bu^\epsilon(t))^-=\{\,(\bx,\by)\not\in SZ^\epsilon(t)\,\},\\
&S(\by,\bx,\bu^\epsilon(t))^+=\{\,(\bx,\by)\in SZ^\epsilon(t)\,\}.
\end{aligned}
\label{decomposedetailsS3}
\end{equation}
In the next section we recover the dynamics in the limit of vanishing horizon with failure zone and softening zone given by \eqref{failure} and \eqref{soften}. The equations \eqref{failure} and \eqref{soften} constitute the hypothesis on the crack structure for the nonlocal model. 
For the loading prescribed here \eqref{failure} and \eqref{soften} naturally emerge and are a consequence of the symmetry of solution  $\bu^\epsilon(\bx,t)$, this is seen in  simulations \cite{Jha-Lipton2020}.

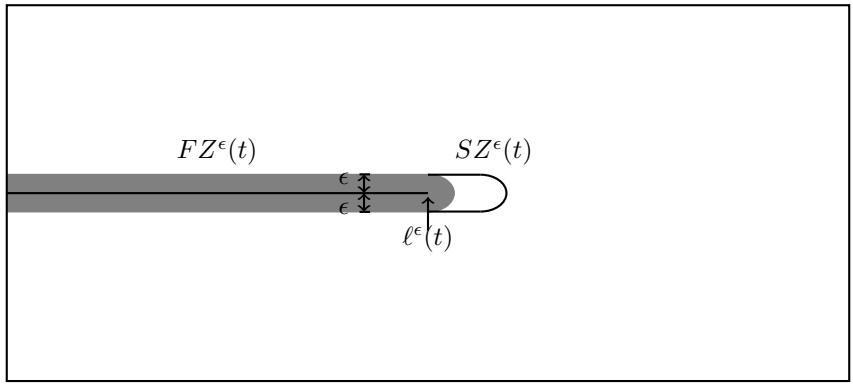
\begin{figure} 
\centering
\begin{tikzpicture}[xscale=0.70,yscale=0.50]

\draw [fill=gray, gray] (-8,-0.5) rectangle (0.0,0.5);






\node [above] at (-4.0,0.5) {$FZ^\epsilon(t)$};

\draw[fill=gray, gray] (0,-0.5) arc (-90:90:0.5)  -- cycle;






\node [above] at (1.25,0.5) {$SZ^\epsilon(t)$};

\draw[-,thick] (1.0,-0.49) arc (-90:90:0.49);

\draw[-,thick] (1.0,-0.49) -- (-0.0,-0.49);

\draw[-,thick] (1.0,0.49) -- (-0.0,0.49);



\draw [-,thick] (-8,0) -- (-0.0,0);




\draw [thick] (-8,-5) rectangle (8,5);












\node [above] at (-1.6,0) {$\epsilon$};

\node [below] at (-1.6,0) {$\epsilon$};

\draw [thick] (-1.3,.5) -- (-1.1,.5);

\draw [<->,thick] (-1.215,.5) -- (-1.215,0);

\draw [<->,thick] (-1.215,-.5) -- (-1.215,0);

\draw [thick] (-1.3,-.5) -- (-1.1,-.5);

\draw [->,thick] (0.0,-1.0) -- (0.0,-0.1);

\node  at (0.0,-1.2) { $\ell^\epsilon(t)$};

\end{tikzpicture} 
\caption{{ \bf The failure zone, failure zone centerline, and softening zone.}}
 \label{P}
\end{figure}

\section{Convergence of nonlocal elastodynamics to  elastic fields in Linear Elastic Fracture Mechanics}\label{ss:crackwaveinteraction}

The crack structure is prescribed by $\ell^\epsilon(t)$  of  \eqref{center} together with \eqref{failure}, and \eqref{soften}, and the elastic fields $\bu^\epsilon$ are solutions of \eqref{energy based model2} and \eqref{idata}.  The crack structure for $\epsilon>0$ is  summarized in the following hypothesis:
\begin{hypothesis}[Crack Structure  for $\epsilon>0$.]\label{hyp1}
The moving domain associated with the defect is prescribed by $\ell^\epsilon(t)$  of  \eqref{center}, and the failure zone and softening zone are given by \eqref{failure}, and \eqref{soften}.
\end{hypothesis}
Given hypothesis \ref{hyp1} we now describe the convergence of  $\bu^\epsilon$ to $\bu^0$ to see that $\bu^0$ satisfies the  boundary value problem for the elastic field of LEFM for a running crack given in \cite{Freund}. 
Recall  $\ell^{\epsilon}(t)$ is monotone increasing with time and bounded so from  Helly's selection theorem we can pass to a subsequence if necessary to assert that $\ell^{\epsilon_n}(t)\rightarrow\ell^0(t)$ point wise for $t\in[0,T]$, where  $\ell^{0}(t)$ is monotone increasing with time and bounded. This delivers the crack motion for the $\epsilon=0$ problem  described by the crack 
\begin{equation}\label{crackforlimitproblem}
\Gamma_t=\{\ell(0)\leq x_1\leq \ell^0(t),\,x_2=0\},\hbox{    } t\in[0,T].
\end{equation}
Here $\tau<t$ implies $\Gamma_\tau\subset\Gamma_t$.
The time dependent domain surrounding the crack is defined as $D_t=D\setminus \Gamma_t$ see figure \ref{target}.

Next we describe the convergence of body force, velocity, and acceleration given by the $\epsilon>0$ initial value problems \eqref{energy based model2} and \eqref{idata} to their $\epsilon=0$ counterparts.
The convergence of the elastic displacement field, velocity field and acceleration field are described  in terms of suitable Hilbert space topologies. The space of strongly measurable functions $\bw\,:\,[0,T]\rightarrow  \dot L^2(D;\mathbb{R}^2)$ that are square integrable in time is denoted by $L^2(0,T;\dot L^2(D;\mathbb{R}^2))$. Additionally we recall the Sobolev space $H^1(D;\mathbb{R}^2)$ with norm 
\begin{equation}\label{h1}
\Vert\bw\Vert_{H^1(D;\mathbb{R}^2)}:=\left(\int_ D\,|\bw|^2\,d\bx+\int_D|\nabla\bw|^2\,d\bx\right)^{1/2}.
\end{equation}
The subspace of $H^1(D;\mathbb{R}^2)$
containing all vector fields orthogonal to the rigid motions with respect to the $L^2(D;\mathbb{R}^2)$ inner product is written
\begin{equation}\label{h85}
\dot H^1(D;\mathbb{R}^2).
\end{equation}
The Hilbert space dual to $\dot H^1(D;\mathbb{R}^2)$ is denoted by $\dot H^1(D;\mathbb{R}^2)'$.  The set of functions strongly square  integrable in time taking values in $\dot H^1(D;\mathbb{R}^2)'$ for $0\leq t\leq T$ is denoted by $L^2(0,T;\dot H^1(D;\mathbb{R}^2)')$.  These Hilbert spaces are well known and related to the wave equation, see \cite{Evans}.

The body force given in \eqref{energy based model2}  is written as $\bb^{\epsilon_n}(t)$  and we state the following lemma.
\begin{lemma}\label{convergingrhs}
There is a positive constant $C$ independent of $\epsilon_n$ and $t\in [0,T]$ such that
\begin{equation}\label{righthandside}
\begin{aligned}
|\langle\bb^{\epsilon_n}(t),\bw\rangle|
\leq C\Vert\bw\Vert_{H^1(D,\mathbb{R}^2)},\hbox{ for all $\epsilon_n>0$ and $\bw \in \dot H^{1}(D,\mathbb{R}^2)$},
\end{aligned}
\end{equation}
where $\langle\cdot,\cdot\rangle$ is the duality paring between $\dot H^{1}(D,\mathbb{R}^2)$ and its Hilbert space dual $\dot H^{1}(D,\mathbb{R}^2)'$. In addition there exists $\bb^0(t)$ such that $\bb^{\epsilon_n}\rightharpoonup\bb^0$ in  $L^2(0,T;\dot H^1(D;\mathbb{R}^2)')$ and
\begin{equation}\label{righthandsidelimit}
\begin{aligned}
\langle\bb^0(t),\bw\rangle=&\langle\bg(t),\bw\rangle:=\int_{\partial D}\,\bg(t)\cdot\bw\,d\sigma,
\end{aligned}
\end{equation}
for all $\bw \in \dot H^{1}(D,\mathbb{R}^2)$, where $\bg(t)$ is defined by \eqref{l2} and $\bg\in H^{-1/2}(\partial D)^2$.
\end{lemma}
The traction force  \eqref{righthandsidelimit}  delivers loading consistent with a mode one crack in the local model given by LEFM.
For ease of exposition we defer the proof of lemma \ref{convergingrhs} as well as proofs of all other theorems introduced here to sections \ref{s:proofofsymmetry} and \ref{s:proofofconvergence}.

Passing to subsequences as necessary we obtain the convergence of the elastic displacement field, velocity field, and acceleration field given by 
\begin{lemma}\label{convergences}
\begin{equation}\label{convgencess}
\begin{aligned}
\bu^{\epsilon_n}\rightarrow \bu^0 & \hbox{  \rm strong in } C([0,T];\dot L^2(D;\mathbb{R}^2))\\
\dot\bu^{\epsilon_n}\rightharpoonup \dot\bu^0 &\hbox{ \rm weakly in  }L^2(0,T;\dot L^2(D;\mathbb{R}^2))\\
\ddot\bu^{\epsilon_n}\rightharpoonup \ddot\bu^0 &\hbox{ \rm weakly in }L^2(0,T;\dot H^1(D;\mathbb{R}^2)'),
\end{aligned}
\end{equation}
where $\dot\bu^0(t)$ and $\ddot\bu^0(t)$ are distributional derivatives in time.
\end{lemma}

With the additional caveat that 
\begin{equation}\label{linfty}
\sup_{[0,T]}\sup_{\epsilon>0}\Vert\bu^\epsilon(t)\Vert_{L^\infty(D,\mathbb{R}^2)}<\infty,
\end{equation}
 the limit evolution $\bu^0(\bx,t)$ is seen to be a special function of bounded deformation $SBD(D)$ for almost all times \cite{CMPer-Lipton3} and \cite{CMPer-Lipton}. We will include \eqref{linfty} in the hypotheses of subsequent theorems when we make use of the fact that $\bu^0$ belongs to $SBD(D)$.
Functions $\bu\in SBD(D)$ belong to  $L^1(D;\mathbb{R}^d)$ (where $d=2$ in this work)  and are approximately continuous, i.e., have Lebesgue limits for almost every $\bx\in D$  given by
\begin{eqnarray}
\lim_{\epsilon\searrow 0}\frac{1}{\omega_2\epsilon^2}\int_{\mathcal{H}_\epsilon(\bx)}\,|\bu(\by)-\bu(\bx)|\,d\by=0, 
\label{approx}
\end{eqnarray}
where $\mathcal{H}_\epsilon(\bx)$ is the ball of radius $\epsilon$ centered at $\bx$ and $\omega_2\epsilon^2$ is its area given in terms of the area of the unit disk $\omega_2$ times $\epsilon^2$. 
The set of points in $D$ which are not points of approximate continuity is denoted by $S_{\bu}$. A subset of these points are given by 
the  jump set $\mathcal{J}_{\bu}$. The jump set  is defined to be the set of points of discontinuity which have two different one sided Lebesgue limits.  One sided Lebesgue limits of  $\bu$ with respect to a direction $\nu_{\bu(\bx)}$ are  denoted by $\bu^-(\bx)$, $\bu^+(\bx)$ and are given by
\begin{equation}
\begin{aligned}
&\lim_{\epsilon\searrow 0}\frac{1}{\epsilon^2\omega_2}\int_{\mathcal{H}^-_\epsilon(\bx)}\,|\bu(\by)-\bu^-(\bx)|\,d\by=0,\\
&\lim_{\epsilon\searrow 0}\frac{1}{\epsilon^2\omega_2}\int_{\mathcal{H}^+_\epsilon(\bx)}\,|\bu(\by)-\bu^+(\bx)|\,d\by=0,
\end{aligned}
\label{approxjump}
\end{equation}
where $\mathcal{H}^-_\epsilon(\bx)$ and $\mathcal{H}^+_\epsilon(\bx)$ are given by the intersection of $\mathcal{H}_\epsilon(\bx)$ with the half spaces $(\by-\bx)\cdot \nu_{\bu(\bx)}<0$ and $(\by-\bx)\cdot \nu_{\bu(\bx)}>0$ respectively. $SBD(D)$ functions have jump sets ${\mathcal{J}}_{\bu}$, that are countably rectifiable. Hence they are described by a countable number of components $K_1,K_2,\ldots$, contained within smooth manifolds, with the exception of a set $K_0$ that has zero $1$ dimensional Hausdorff measure  
\cite{AmbrosioCosicaDalmaso}. 
The one dimensional Hausdorff  measure of ${\mathcal J}_{\bu}$ agrees with the one dimensional Lesbegue measure and $\mathcal{H}^{1}({\mathcal J}_{\bu})=\sum_i\mathcal{H}^{1}(K_i)$.
The  strain  of a displacement $\bu$ belonging to $SBD(D)$, written as  $\mathcal{E}_{ij} \bu^0(t)=(\partial_{x_i}\bu^0_j+\partial_{x_j}\bu^0_i)/2$, is a generalization of the classic local strain tensor and is related to the 
nonlocal strain $S(\by,\bx,\bu^0)$  by 
\begin{equation}
\lim_{\epsilon\rightarrow 0}\frac{1}{\epsilon^2\omega_2}\int_{\mathcal{H}_\epsilon(\bx)}|{S(\by,\bx,\bu^0)}-\mathcal{E}\bu^0(\bx) \be\cdot \be|\,d\by=0,
\label{equatesandE}
\end{equation}
for almost every $\bx$ in $D$ with respect to $2$-dimensional Lebesgue measure $\mathcal{L}^2$. 
The symmetric part of the distributional derivative of $\bu$, $E \bu=1/2(\nabla \bu+\nabla \bu^T)$ for $SBD(D)$ functions is a $2\times 2$ matrix valued Radon measure with absolutely continuous part with respect to two dimensional Lesbesgue measure described by the density $\mathcal{E}\bu$ and singular part described by the jump set \cite{AmbrosioCosicaDalmaso} and
\begin{eqnarray}
\langle E u,\Phi\rangle=\int_D\,\sum_{i,j=1}^d\mathcal{E}u_{ij}\Phi_{ij}\,d\bx+\int_{\mathcal{J}_{u}}\,\sum_{i,j=1}^d(\bu^+_i - \bu^-_i)\bn_j\Phi_{ij}\,d\mathcal{H}^{1},
\label{distderiv}
\end{eqnarray}
for every continuous, symmetric matrix valued test function $\Phi$. In the sequel we will write $[\bu]=\bu^+-\bu^-$.

The limit dynamics and LEFM energy are expressed in terms of elastic moduli $\lambda$ and $\mu$ and fracture toughness $\mathcal{G}$. These are calculated directly from the nonlocal potential \eqref{tensilepot}. Here we have taken the choice $\Psi(r)=h(r^2)$ and the elastic moduli are given by
\begin{equation}\label{calibrate1}
\mu=\lambda=M\frac{1}{4} h'(0) \,, 
\end{equation}
where the constant $M=\int_{0}^1r^{2}J(r)dr$. The elasticity tensor is given by
\begin{equation}\label{tensor}
\begin{aligned}
\mathbb{C}_{ijkl}=2\mu\left(\frac{\delta_{ik}\delta_{jl}+\delta_{il}\delta_{jk}}{2}\right)+\lambda\delta_{ij}\delta_{kl},
\end{aligned}
\end{equation}
and
\begin{eqnarray}
\mathcal{G}_c=\, \frac{4}{\pi}\int_{0}^1h(S_+)r^2J(r)dr.
\label{epsilonfracttough}
\end{eqnarray}

The limit evolution has a bounded Griffith surface energy and elastic energy given by 
\begin{eqnarray}
\int_{D}\,\mu |\mathcal{E} \bu^0(t)|^2+\frac{\lambda}{2} |{\rm div}\,\bu^0(t)|^2\,d\bx+\mathcal{G}\mathcal{H}^1(\mathcal{J}_{\bu^0(t)})\leq C, 
\label{LEFMbound}
\end{eqnarray}
for $0\leq t\leq T$,
where $\mathcal{J}_{\bu^0(t)}$ denotes the evolving jump set inside the domain $D$, across which the  displacement $\bu^0$ has a jump discontinuity and  $\mathcal{H}^1$ is one dimensional Hausdorff measure, see \cite{CMPer-Lipton3} and \cite{CMPer-Lipton}. 
Because $\bu^0$ has bounded energy \eqref{LEFMbound} we see that $\bu^0$ also belongs to $SBD^2(D)$. Here $SBD^2(D)$ is the set of $SBD(D)$ functions with square integrable strain $\mathcal{E}\bu$ and jump set with bounded $\mathcal{H}^1$ measure. It has been recently shown in \cite{ContiFocardiIurlano} that for $\bu \in SBD^2(D)$ that 
\begin{equation}\label{reg}
\mathcal{H}^1(S_{\bu}\setminus \mathcal{J_{\bu}})=0.
\end{equation}
It is remarked that the equality $\lambda=\mu$  appearing in \eqref{calibrate1} is a consequence of the central force nature of the nonlocal interaction mediated by \eqref{tensilepot}.   While non-central force potentials can deliver a larger class of energy-volume-shape change relations  \cite{States} a central force potential is been chosen  to illustrate the ideas. 

The symmetry of the limit displacement $\bu^0$ as an element of $SBD^2(D)$ follows from the symmetry of $\bu^\epsilon$. 
\begin{theorem}\label{symmetry2}
The displacement $\bu^0$ is in $ SBD^2(D)$ for a.e. $t\in (0,T)$ and its first component denoted by $u^0_1(x_1,x_2)$ is even with respect to the $x_2=0$ axis and the second component of the displacement denoted by $u^0_2(x_1,x_2)$ is odd with respect to the $x_2=0$ axis and $u^0_2(x_1,x_2)=0$, $\mathcal{H}^1$ a.e. for $\{\ell^0(t)< x_1< a,\, x_2=0\}$. The jump set of $\bu^0$ is contained inside the crack $\Gamma_t$,  $t\in[0,T]$.
\end{theorem}

The global description of  $\ddot\bu^0(t)$ can be further specified in terms of suitable Sobolev spaces posed over time dependent domains.
For  $\ell(0)=\ell^0(0)<\ell^0(t)$  monotonicity implies $0<t<T$.  
We choose $0\leq\beta<\ell^0(t)-\ell(0)$ and introduce $ D_\beta(t)=D\setminus\{\ell(0)\leq x_1\leq\ell^0(t)-\beta;\,x_2=0\}$. It is evident that $D_\beta(t)\subset D_t$ and its boundary is denoted by $\partial D_\beta(t)$. The subsets of the boundary $\partial D_\beta{(t)}$ bordering the domains $\{\bx\in D_\beta(t):\, \pm x_2\geq 0\}$ are denoted by $\partial D_{\beta}^\pm(t)$.  
The layer $L^+_\beta(t)$  adjacent to $\partial D^+_\beta(t)$ is defined to be the region inside the solid and dashed contours portrayed in figure \ref{L-upper}. The dashed contour interior to $D_\beta(t)$ is denoted by $\partial L^+$.
For $0<t<T$ set
\begin{equation}\label{cracklayer}
W^+(D_\beta(t))=\left\{\bw\in H^1(L^+_\beta(t),\mathbb{R}^2)\text{ and }\gamma \bw =0 \text{ on }  \partial L^+,\, \bw\,\,\hbox{extended by $0$ to $D_\beta(t)$}\right\},
\end{equation}
here $\gamma$ is the trace operator mapping functions in $H^1(L^+_\beta(t),\mathbb{R}^2)$ to functions defined on the boundary.  The Hilbert space dual to $W^+(D_\beta(t))$ is denoted by $W^+(D_\beta(t))'$. We introduce the layer $L^-_\beta(t)$ adjacent to the boundary $\partial D_\beta^-(t)$ and the boundary of the layer internal to $D_\beta(t)$ is denoted by $\partial L^-$.
The analogous space $W^-(D_\beta(t))$ is given by
\begin{equation}\label{cracklayerlower}
W^-(D_\beta(t))=\left\{\bw\in H^1(L^-_\beta(t),\mathbb{R}^2)\text{ and }\gamma \bw =0 \text{ on }  \partial L^-,\, \bw\,\,\hbox{extended by $0$ to $D_\beta(t)$}\right\},
\end{equation}
with dual  $W^-(D_\beta(t))'$.

For any $\tau\in(0,T)$ let $\bu^0_\tau$ be the restriction of $\bu^0$ to $\tau<t<T$. 
Then we have the following theorem.
\begin{theorem}\label{limitspace}   For all $\tau\in(0,T)$, $\ddot\bu_\tau^0(\bx,t)$ belongs to $W^{\pm}(D_\beta(\tau))'$ for almost all $t\in(\tau,T)$ and
\begin{equation}\label{convgencessrestrict}
\begin{aligned}
\ddot\bu^{\epsilon_n}\rightharpoonup \ddot\bu^0_\tau &\hbox{ \rm weakly in }L^2(\tau,T;W^\pm(D_\beta(\tau))').
\end{aligned}
\end{equation}
\end{theorem}
Since $\ddot\bu^0_\tau$ belongs to $W^{\pm}(D_\beta(\tau))'$  we introduce the 
the normal traction $\mathbb{C}\,\mathcal{E}\bu^0\bn$  defined on the crack lips for $(\tau,T)$ and $\partial D$ in the generalized sense \cite{McLean}.  In order to describe the generalized traction we introduce trace spaces compatible with the crack geometry.
For $t\in[0,t]$ we introduce the weight defined on $\partial D_{\beta}^\pm(t)$ given by
\begin{align}\label{weight}
\alpha_\pm(x_1,x_2,\beta)= \begin{cases}
\min\{1,\sqrt{(\ell^0(t)-\beta-x_1)}\}, &\qquad \text{on } x_2=0 \\
\min\{1,\sqrt{\pm x_2}\}, &\qquad \text{on } x_1=a,\,\,\pm x_2>0\\
1,&\qquad \text{otherwise} .
\end{cases}
\end{align}
and the trace spaces $H_{00}^{1/2}(\partial D_{\beta}^\pm(t))^2$ given in \cite{LM} are defined by all  functions $\bw$ in $H^{1/2}(\partial D_{\beta}^\pm(t))^2$ with
\begin{equation}
\int_{\partial D_{\beta,t}^\pm}|\bw(\bx)|^2\alpha_\pm^{-1}(\bx,\beta)ds<\infty.
\end{equation}
The dual to $H_{00}^{1/2}(\partial D_{\beta}^\pm(t))^2$ is ${H}_{00}^{-1/2}(\partial D_{\beta}^\pm(t))^2$. This type of trace space is employed for problems of mechanical contact in \cite{KO}, see also \cite{Schwab}. The trace operator $\gamma$ is a continuous linear map from $W^\pm(D_\beta(t))$ onto ${H}_{00}^{1/2}(\partial D_{\beta}^\pm(t))^2$, see \cite{LM}. Additionally the trace operator $\gamma$ is a continuous linear map from $H^{1}(D,\mathbb{R}^2)$ onto $H^{1/2}(\partial D)^2$. 

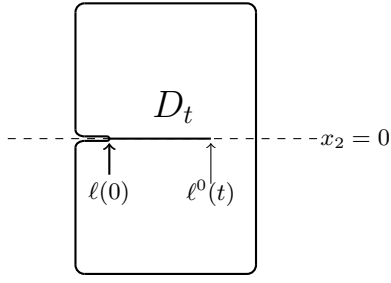
\begin{figure} 
\centering
\begin{tikzpicture}[xscale=0.60,yscale=0.60]

%
%
\draw [-,thick] (-1.80,0.05) -- (-1.3,0.05);

\draw [-,thick] (-1.80,-0.05) -- (-1.3,-0.05);


\draw [-,thick] (-1.3,0.05) to [out=0, in=90] (-1.25 ,0);

\draw [-,thick] (-1.25,0) to [out=-90, in=0] (-1.3 ,-0.05);


\draw [-,thick] (-2,0.2) to [out=-90, in=180] (-1.8 ,0.05);

\draw [-,thick] (-2,-0.2) to [out=90, in=180] (-1.8 ,-0.05);

\draw [-,thick] (-2,2.8) to [out=90, in=180] (-1.8 ,3.0);

\draw [-,thick] (-2,-2.8) to [out=-90, in=180] (-1.8 ,-3.0);

\draw [-,thick] (1.8,3.0) to [out=0, in=90] (2.0 ,2.8);

\draw [-,thick] (1.8,-3.0) to [out=0, in=-90] (2.0 ,-2.8);
%
%
%
\draw [-,dashed] (-3.5,0.0) -- (3.5,0.0);

\node  at (4.2,0.0) {\small $x_2=0$};
%
%

\draw [-,thick] (-1.25,0) -- (1.0,0);

\draw [->] (1.0,-1.0) -- (1.0,-0.1);

\node  at (1.0,-1.2) {\small $\ell^0(t)$};

\draw [<-,thick] (-1.25,-0.1) -- (-1.25,-0.8);

\node  at (-1.25,-1.2) {\small $\ell(0)$};


%
%
\draw [-,thick] (-2,2.8) -- (-2,0.2);
\draw [-,thick] (-2,-2.8) -- (-2,-0.2);
\draw [-,thick] (-1.8,-3) -- (1.8,-3);
\draw [-,thick] (2,2.8) -- (2,-2.8);
\draw [-,thick] (1.8,3) -- (-1.8,3);
%
%
%
%
\node [right] at (-0.5,0.7) {{\Large $D_t$}};

\end{tikzpicture} 
\caption{{ \bf Single-edge-notch and crack corresponding to $\epsilon =0$ limit.}}
 \label{target}
\end{figure}

In what follows the duality bracket for Hilbert spaces $H$ and their dual $H'$ is defined by $\langle\cdot,\cdot\rangle$, where the first argument is an element of $H'$ and the second an element of $H$. 
The generalized traction   $\mathbb{C}\,\mathcal{E}\bu^0\bn$ on $\partial D$ is introduced as an element of $H^{-1/2}(\partial D)^2$. For this case we have suitable integration by parts formulas given by the following two lemmas. 
\begin{lemma}\label{generaltractbdry}
Since $\ddot \bu^0$ belongs to $H^1(D,\mathbb{R}^2)'$ and $\bu^0$ is in $SBD^2(D)$ then the generalized traction $\mathbb{C}\,\mathcal{E}\bu^0\bn$  is uniquely defined as an  element of ${H}^{-1/2}(\partial D)^2$ on the boundary $\partial D$ is given by
\begin{equation}\label{generalt1}
\langle \mathbb{C}\,\mathcal{E}\bu^0\bn,\gamma\bw\rangle=\int_{D}\mathbb{C}\,\mathcal{E}\bu^0:\mathcal{E}{\bw}\,d\bx+\rho\langle\ddot\bu^0,\bw\rangle,
\end{equation}
for all test functions $\bw$ in $H^1(D,\mathbb{R}^2)$
is uniquely defined. 
\end{lemma}

\begin{lemma}\label{generaltractoncrack}
Since $\ddot \bu_\tau^0(t)$ belongs to $W^\pm(D_\beta(\tau))'$  for a.e., $t\in (\tau,T)$ and $\bu^0(t)$ is in $SBD^2(D)$ the generalized tractions  $\mathbb{C}\,\mathcal{E}\bu^0(t)\,\bn^\pm$  are uniquely defined as  elements of ${H}_{00}^{-1/2}(\partial D^\pm_{\beta}(\tau))^2$ on the upper and lower sides of the crack $\Gamma_{t}$ by
\begin{equation}\label{generalt2}
\langle \mathbb{C}\,\mathcal{E}\bu^0(t)\,\bn^\pm,\gamma\bw\rangle=\int_{L^{\pm}_\beta(\tau)}\mathbb{C}\,\mathcal{E}\bu^0(t):\mathcal{E}{\bw}\,d\bx+\rho\langle\ddot\bu_\tau^0(t),\bw\rangle,
\end{equation}
for all test functions $\bw$ in $W^\pm(D_\beta(\tau))$ and a.e., $t\in (\tau,T)$. 

\end{lemma}
Lemmas \ref{generaltractbdry} and \ref{generaltractoncrack} are proved in section \ref{s:proofofconvergence}.

The global dynamics for $\bu^0(\bx,t)$ is given by the following theorem.
\begin{theorem}\label{momentumlim} The limit displacement field $\bu^0$ satisfies
\begin{equation}\label{mo}
\rho\ddot\bu^0=div\left(\mathbb{C}\mathcal{E}\bu^0\right) 
\end{equation}
as elements of $H^{-1}(D,\mathbb{R}^2)$, for a.e., $t\in(0,T)$ and
\begin{equation}\label{boundary tracesAB}
\begin{aligned}
\mathbb{C}\,\mathcal{E}\bu^0\bn=\bg&\hbox{ on $\partial D$},\\
\end{aligned}
\end{equation}
where the traction $\bg$ is given by \eqref{l2} and equality holds as elements of $H^{-1/2}(\partial D)^2$ for a.e., $t\in(0,T)$. Moreover there is zero traction on the upper and lower sides of the crack $\Gamma_\tau$, $\tau\in(0,T)$, this is given by
\begin{equation}\label{boundary tracesG}
\begin{aligned}
\mathbb{C}\,\mathcal{E}\bu^0(t)\bn^\pm=0,\,\,\,\hbox{for $\{\ell(0)<x_1\leq \ell^0(\tau)-\beta;\,\,x_2=0\}$}
\end{aligned}
\end{equation}
as  elements of  $H_{00}^{-1/2}(\partial D_{\beta}^\pm(\tau))^2$ for a.e., $t\in(\tau,T)$, for all $\beta\in (0,\ell^0(\tau)-\ell(0))$.
\end{theorem}
\noindent Here the normal tractions \eqref{boundary tracesAB} and \eqref{boundary tracesG} are defined in the generalized sense \eqref{generalt1}, \eqref{generalt2} respectively.
To summarize theorem \ref{momentumlim} delivers the global description of the displacement fields inside the cracking body. Together they deliver the elastodynamic equations and homogeneous traction boundary conditions on the crack faces given in  LEFM  \cite{Freund}, \cite{RaviChandar}, \cite{Anderson},  and \cite{Slepian}.

The field $\bu^0(t,\bx)$ is seen to be a weak solution of the wave equation on $D_t$ for $t\in[0,T]$. We begin with the definition of weak solution of the wave equation on time dependent domains introduced in \cite{DalToader}. Neumann boundary conditions are considered and the space  $\dot H^1(D_t,\mathbb{R}^2)=H^1(D_t,\mathbb{R}^2)\cap\dot L^2(D,\mathbb{R}^2)$ is introduced. Set $V_t=\dot H^1(D_t,\mathbb{R}^2)$,  $V_t^*=\dot H^1(D_t,\mathbb{R}^2)'$ for $t\in[0,T]$, and  $H=\dot L^2(D,\mathbb{R}^2)$. Recall $\Gamma_s\subset\Gamma_t$ when $0\leq s\leq t\leq T$ and $\mathcal{H}^1(\Gamma_T)< a-\ell(0)$.

\begin{definition}\label{defweakspaces}
\cite{DalToader} $\mathcal{V}$ is the space of functions $\bv\in  L^2(0,T;\, V_T)\cap H^1(0,T;\,H) $ such that $\bv(t) \in V_t$ for a.e. $t\in (0,T)$. It is a Hilbert space with scalar product given by
\begin{equation}
\label{innerprodspacetime}
(\bu,\bv)_\mathcal{V}=(\bu,\bv)_{ L^2(0,T;\, V_T)}+(\dot\bu,\dot\bv)_{ L^2(0,T;\,H)},
\end{equation}
where $\dot\bu$ and $\dot\bv$ denote distributional derivatives with respect to $t$.
\end{definition}

\begin{definition}\label{defweaksoln}
\cite{DalToader} 
Given $\bg(t)$ defined by \eqref{l2} the displacement $\bu$ is said to be a weak solution of the wave equation
\begin{align}\label{formulation}
\begin{cases}
\rho\ddot\bu(t)+{\rm div}(\mathbb{C}\mathcal{E}\bu(t))=0 \\
\mathbb{C}\mathcal{E}\bu(t)\bn=\bg(t), \text{ on }\partial D\\
\bu(t)\in V_t
\end{cases}
\end{align}
on the time interval $[0,T]$ if $\bu\in \mathcal{V}$ and
\begin{equation}
\label{weaksoln}
-\int_0^T\rho\int_D\dot\bu(t)\cdot\dot\varphi(t)\,d\bx\,dt+\int_0^T\int_D\mathbb{C}\mathcal{E}\bu(t):\mathcal{E}\varphi(t)\,d\bx\,dt=\int_0^T\int_{\partial D}\bg(t)\cdot\varphi(t)\,d\sigma\,dt
\end{equation}
for every $\varphi\in\mathcal{V}$ with $\varphi(T)=\varphi(0)=0$.
\end{definition}

\begin{theorem}
\label{zerohorizonweaksoln}
If the crack tip $\ell^0(t)$ is continuous and strictly increasing for $t\in[0,T]$ then the limit displacement $\bu^0$ is a weak solution of the wave equation on $D_t$ for $t\in[0,T]$ given by  definition \ref{defweaksoln}.
\end{theorem}
Theorem \ref{zerohorizonweaksoln} establishes the link between the nonlocal theory and the theory of the wave equation on time dependent domains \cite{DalToader}.  Here the choice of test functions delivers a variational description of vanishing normal traction for the solution of the weak formulation. If one assumes a more general crack structure for the nonlocal problem then a connection to the local problem for more general time dependent domains can be made, this is discussed in the conclusion. 


\section{Existence and uniqueness of nonlocal elastodynamics}\label{existenceuniqueness}
We assert the existence and uniqueness for a solution $\bu^\epsilon(\bx,t)$ of the nonlocal evolution with the balance of momentum given in strong form \eqref{energy based model2}.

\begin{figure} 
\centering
\begin{tikzpicture}[xscale=0.60,yscale=0.60]

%
%
\draw [-,thick] (-1.80,0.05) -- (-1.3,0.05);

\draw [-,thick] (-1.80,-0.05) -- (-1.3,-0.05);

\draw [-,thick] (-1.3,-0.05) -- (-1.3,0.05);

\draw [-,thick] (-1.3,0.05) to [out=0, in=90] (-1.25 ,0);



\draw [-,thick] (-1.25,0) -- (0.0,0);


\draw [<-,thick] (-1.25,-0.1) -- (-1.25,-0.8);

\node  at (-1.25,-1.2) {\small $\ell(0)$};

\draw [<-,thick] (0.0,-0.1) -- (0.0,-2.2);

\node  at (0.0,-2.4) {\small $\ell^0(t)-\beta$};




\draw [-,thick] (-2,0.2) to [out=-90, in=180] (-1.8 ,0.05);

\draw [-,thick] (-2,-0.2) to [out=90, in=180] (-1.8 ,-0.05);

\draw [-,thick] (-2,2.8) to [out=90, in=180] (-1.8 ,3.0);

\draw [-,thick] (-2,-2.8) to [out=-90, in=180] (-1.8 ,-3.0);

\draw [-,thick] (2.5,3.0) to [out=0, in=90] (2.7 ,2.8);

\draw [-,thick] (1.8,-3.0) to [out=0, in=-90] (2.7 ,-2.8);


%
%
\draw [-,thick] (-2,2.8) -- (-2,0.2);

\draw [-,thick] (2.7,-2.8) -- (2.7,0.0); 

\draw [-,thick] (-2,-2.8) -- (-2,-0.2);
\draw [-,thick] (-1.8,-3) -- (1.8,-3);

\draw [-,thick] (2.7,2.8) -- (2.7,0);
\draw [-,thick] (-1.8,3) -- (2.5,3);
%
%

\draw [-,thick,dashed] (0.0,0.0) -- (0.0,1.5);

\draw [-,thick,dashed] (-1.25,1.5) -- (0.0,1.5);

\draw [-,thick,dashed] (-1.25,1.5) -- (-1.25,2.0);

\draw [-,thick,dashed] (-1.25,2.0) -- (1.60,2.0);

\draw [-,thick,dashed] (1.60,2.0) -- (1.60,0.0);

\draw [-,thick,dashed] (1.60,0.0) -- (2.7,0.0);

%
%
\node [right] at (-2.3,0.8) {{ $L_\beta^+(t)$}};

\node [right] at (-0.10,-0.5) {{ $D_\beta(t)$}};





\end{tikzpicture} 
\caption{{ \bf  Domain  $L^+_\beta(t)$ adjacent to $\partial D_\beta^+(t)$. The boundary of $L^+_\beta(t)$ interior to $D_\beta(t)$ is denoted by the dashed line.}}
 \label{L-upper}
\end{figure}
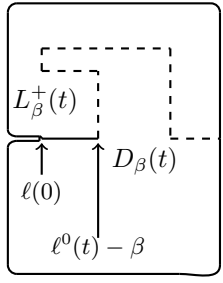

\begin{theorem}\label{existenceuniquness}{
\bf Existence and uniqueness of the nonlocal evolution.}
The initial value problem given by \eqref{energy based model2} and \eqref{idata} has a unique solution $\bu(\bx,t)$ such that for every $t\in [0,T]$, $\bu$ takes values in $\dot L^2(D;\mathbb{R}^2)$ and belongs to the space $C^2([0,T];\dot L^2(D;\mathbb{R}^2))$. 
\end{theorem}

The proof of this proposition follows from the Lipschitz continuity of $\mathcal{L}^\epsilon(\bu^\epsilon)(\bx,t)+\bb(\bx,t)$ as a function of  $\bu^\epsilon$ with respect to the $L^2(D;\mathbb{R}^2)$ norm and the Banach fixed point theorem, see e.g. \cite{CMPer-Lipton4}.
It is pointed out that $SZ^\epsilon$ describes an unstable  phase of the material however because the peridynamic force is a uniformly Lipschitz function on $\dot L^2(D;\mathbb{R}^2)$  the model can be viewed as an ODE for vectors in $\dot L^2(D;\mathbb{R}^2)$ and is well posed.

\section{Symmetry of the limiting elastic displacement field}\label{s:proofofsymmetry}

In this section theorem \ref{symmetry2} is established. To prove theorem \ref{symmetry2} the following lemma is used. 
\begin{lemma}
\label{measures}
\begin{equation}\label{convinmeasure}
\begin{aligned}
&\lim_{\epsilon_n\rightarrow 0}\frac{1}{{\epsilon_n}^2\omega_2} \int_D\int_{{\mathcal H}_{\epsilon_n}(\bx)\cap D}\frac{|\by-\bx|}{\epsilon_n}J^{\epsilon_n}(|\by-\bx|)S(\by,\bx,\bu^{\epsilon_n}(t))^-d\by\,\varphi(\bx)\,d\bx
\\
&=\int_D div\,\bu^0(\bx,t)\varphi(\bx)\,d\bx\\
&\lim_{\epsilon_n\rightarrow 0}\frac{1}{{\epsilon_n}^2\omega_2} \int_{SZ^{\epsilon_n}}\int_{{\mathcal{H}}_{\epsilon_n}(\bx)\cap D}\frac{|\by-\bx|}{\epsilon_n}J^{\epsilon_n}(|\by-\bx|)S(\by,\bx,\bu^{\epsilon_n}(t))^+d\by\,\varphi(\bx)\,d\bx\\
 &=C\int_{{\mathcal{J}}_{\bu^0(t)}}[\bu^0(\bx,t)]\cdot \bn\,\varphi(\bx)d{\mathcal{H}}^1(\bx)
\end{aligned}
\end{equation}
for all scalar test functions $\varphi$ that are differentiable with support in $D$. Here $[\bu^0(\bx,t)]$ denotes the jump in displacement across $\mathcal{J}_{\bu^0(t)}$ and $\bn$ is the unit normal to $\mathcal{J}_{\bu^0(t)}$ and points in the vertical direction $\be^2$, and $C=\omega_2\int_0^1r^2dr$.

\end{lemma}

\begin{proof}[Proof of Lemma \label{convinmeasure}]
It is convenient to make the change of variables $\by=\bx+\epsilon{\xi}$ where ${\xi}$ belongs to the unit disk at the origin $\mathcal{H}_1(0)=\{|{\xi}|<1\}$ and $\be=\xi/|\xi|$. The strain is written
\begin{equation}
\begin{aligned}
\frac{\bu^\epsilon(\bx+\epsilon\xi)-\bu^\epsilon(\bx)}{\epsilon|\xi|}&:=D_{\be}^{\epsilon|\xi|}\bu^\epsilon,  \hbox{  and}\\
S(\by,\bx,\bu^\epsilon(t))&=D_{\be}^{\epsilon|\xi|}\bu^\epsilon\cdot\be, 
\end{aligned}
\label{straiin}
\end{equation}
and for infinitely differentiable scalar valued functions $\varphi$ and vector valued functions $\bw$ bounded and continuous on $D$ we have
\begin{equation}\label{testgrad}
\lim_{\epsilon\rightarrow 0}D_{-\be}^{\epsilon|\xi|}\varphi=-\nabla\varphi\cdot\be,
\end{equation}
and
\begin{eqnarray}\label{teststrain}
\lim_{\epsilon\rightarrow 0}D_{\be}^{{\epsilon}|\xi|} \bw\cdot \be=\mathcal{E} \bw \,\be\cdot \be \label{grad}
\end{eqnarray}
where the convergence is uniform in $D$.
We now recall $S(\by,\bx,\bu^\epsilon(t))^-=D_{\be}^{\epsilon|\xi|}\bu^\epsilon\cdot\be^-$ defined by \eqref{decomposedetailsS}.
We extend $D_{\be}^{\epsilon|\xi|}\bu^\epsilon\cdot\be^-$ by zero when $\bx\in D$ and $\bx+\epsilon\xi\not\in D$ and
\begin{equation}
\label{changevbls}
\begin{aligned}
&\frac{1}{{\epsilon_n}^2\omega_2} \int_D\int_{{\mathcal H}_{\epsilon_n}(\bx)\cap D}\frac{|\by-\bx|}{\epsilon_n}J^{\epsilon_n}(|\by-\bx|)|S(\by,\bx,\bu^{\epsilon_n}(t))^-|^2d\by\,d\bx\\
&=\int_{D\times\mathcal{H}_1(0)}|\xi| J(|\xi|)|(D_{\be}^{{\epsilon_n}|\xi|}\bu^{\epsilon_n}\cdot \be)^-|^2\,d\xi\,d\bx.
\end{aligned}
\end{equation}
Then as in inequality (6.73) of \cite{CMPer-Lipton} we have that
\begin{equation}\label{L2bound}
\begin{aligned}
\int_{D\times\mathcal{H}_1(0)}|\xi| J(|\xi|)|(D_{\be}^{{\epsilon_n}|\xi|}\bu^{\epsilon_n}\cdot \be)^-|^2\,d\xi\,d\bx<C,
\end{aligned}
\end{equation}
for all $\epsilon_n>0$.
From this we can conclude there exists a function $g(\bx,\xi)$ such that a subsequence 
\begin{equation}\label{weakdiffquo}
D_{\be}^{{\epsilon_n}|\xi|}\bu^{\epsilon_n}\cdot \be^-\rightharpoonup g(\bx,\xi)
\end{equation}
 converges weakly in $L^2(D\times\mathcal{H}_1(0),\mathbb{R})$ where the $L^2$ norm and inner product are with respect to the weighted measure $|\xi|J(|\xi|)d\xi d\bx $. Now for any positive number $\eta$ and any subset $D'$ compactly contained in $D_t$ we can argue as in (\cite{CMPer-Lipton} proof of lemma 6.6) that $g(\bx,\xi)=\mathcal{E}\bu^0\be\cdot\be$ for all points in $D'$ with $dist(D',\partial D_t)>\eta$. Since $D'$ and $\eta$ is arbitrary we get that
\begin{equation}\label{weaklim}
g(\bx,\xi)=\mathcal{E}\bu^0\be\cdot\be
\end{equation}
almost everywhere in  $D$. Additionally for any smooth scalar test function $\varphi(\bx)$ with compact support in $D$ straight forward computation gives 
\begin{equation}
\label{identify}
\begin{aligned}
&\lim_{\epsilon_n\rightarrow0}\int_{D\times\mathcal{H}_1(0)}|\xi| J(|\xi|)D_{\be}^{{\epsilon_n}|\xi|}\bu^{\epsilon_n}\cdot \be^-\,d\xi\varphi(\bx)\,d\bx\\
&=\int_{D\times\mathcal{H}_1(0)}|\xi| J(|\xi|)g(\bx,\xi)\,d\xi\varphi(\bx)\,d\bx\\
&=
\int_{D\times\mathcal{H}_1(0)}|\xi| J(|\xi|)\mathcal{E}\bu^0(\bx)\be\cdot\be\,d\xi\varphi(\bx)\,d\bx\\
&=C\int_D\,div\bu^0(\bx)\varphi(\bx)d\bx,
\end{aligned}
\end{equation}
Here $C=\omega_2\int_0^1r^2\,J(r)\,dr$ and we have used
\begin{equation}\label{convegDspoon}
\begin{aligned}
\frac{1}{\omega_2}\int_{\mathcal{H}_1(0)}|\xi|J(|\xi|)\be_ie_j\,d\xi=\delta_{ij}\int_0^1r^2J(r)\,dr.
\end{aligned}
\end{equation}
On the other hand for any smooth test function $\varphi$ with compact support in $D$ we can integrate by parts and use \eqref{testgrad} to write
\begin{equation}\label{L1}
\begin{aligned}
&\lim_{\epsilon_n\rightarrow0}\int_{D\times\mathcal{H}_1(0)}|\xi| J(|\xi|)D_{\be}^{{\epsilon_n}|\xi|}\bu^{\epsilon_n}\cdot \be\varphi(\bx)\,d\xi\,d\bx\\
&=\lim_{\epsilon_n\rightarrow0}\int_{D\times\mathcal{H}_1(0)}|\xi| J(|\xi|)D_{-\be}^{{\epsilon_n}|\xi|}\varphi(\bx)\bu^{\epsilon_n}\cdot \be,d\xi\,d\bx\\
&=-\int_{D\times\mathcal{H}_1(0)}|\xi| J(|\xi|)\bu^0\cdot \be\,\nabla\varphi(\bx)\cdot\be\,d\xi\,d\bx\\
&=-C\int_{D}\bu^0\cdot \nabla\varphi(\bx)\,d\bx\\
&=C\int_D\,tr{E\bu^0}\varphi(\bx)\,d\bx,
\end{aligned}
\end{equation}
where $E\bu^0$ is the strain of the $SBD^2$ limit displacement $\bu^0$.
Now since $\bu^0$ is in $SBD$ its weak derivitave satisfies \eqref{distderiv} and it follows on choosing $\Phi_{ij}={\delta}_{ij}\varphi$ that 
\begin{equation}\label{radon}
\begin{aligned}
\int_D\,tr{E\bu^0}\varphi\,d\bx=\int_D\,div\bu^0\varphi\,d\bx+\int_{\mathcal{J}_{\bu^0(t)}}[\bu^0]\cdot \bn\,\varphi\,d\mathcal{H}^1(\bx),
\end{aligned}
\end{equation}
and 
\begin{equation}\label{diff}
\begin{aligned}
&\int_{D\times\mathcal{H}_1(0)}|\xi| J(|\xi|)D_{\be}^{{\epsilon_n}|\xi|}\bu^{\epsilon_n}\cdot \be\,d\xi\varphi(\bx)\,d\bx\\
&=\int_{D\times\mathcal{H}_1(0)}|\xi| J(|\xi|)(D_{\be}^{{\epsilon_n}|\xi|}\bu^{\epsilon_n}\cdot \be)^-d\xi\varphi(\bx)\, d\bx\\
&+ \int_{D\times\mathcal{H}_1(0)}|\xi| J(|\xi|)(D_{\be}^{{\epsilon_n}|\xi|}\bu^{\epsilon_n}\cdot \be)^+d\xi \varphi(\bx)\,d\bx
\end{aligned}
\end{equation}
to conclude
\begin{equation}\label{radonlim}
\begin{aligned}
&\lim_{\epsilon_n \rightarrow 0}\int_{D\times\mathcal{H}_1(0)}|\xi| J(|\xi|)(D_{\be}^{{\epsilon_n}|\xi|}\bu^{\epsilon_n}\cdot \be)^+d\xi \varphi(\bx)\,d\bx\\
&=C\int_{\mathcal{J}_{\bu^0(t)}}[\bu^0]\cdot \bn\,\varphi\,d\mathcal{H}^1(\bx).
\end{aligned}
\end{equation}
On changing variables we obtain the identities:
\begin{equation}\label{s1}
\begin{aligned}
&\lim_{\epsilon_n \rightarrow 0}\frac{1}{\epsilon_n^2}\int_{D}\int_{\mathcal{H}_{\epsilon_n}(\bx)}\frac{|\by-\bx|}{\epsilon_n} J^{\epsilon_n}(|\by-\bx|)
S(\by,\bx,\bu^{\epsilon_n}(t))^+\,d\by \,\varphi(\bx)\,d\bx\\
&=C\int_{\mathcal{J}_{\bu^0(t)}}[\bu^0]\cdot \bn\,\varphi\,d\mathcal{H}^1(\bx).
\end{aligned}
\end{equation}
and
\begin{equation}
\label{s2}
\begin{aligned}
&\lim_{\epsilon_n\rightarrow0}\frac{1}{\epsilon_n^2}\int_{D}\int_{\mathcal{H}_{\epsilon_n}(\bx)}\frac{|\by-\bx|}{\epsilon_n} J^{\epsilon_n}(|\by-\bx|)
S(\by,\bx,\bu^{\epsilon_n}(t))^-\,d\by \,\varphi(\bx)\,d\bx\\
&=C\int_D\,div\bu^0(\bx)\varphi(\bx)d\bx,
\end{aligned}
\end{equation}
and  lemma \ref{measures} is proved. 
\end{proof}

To prove theorem \ref{symmetry2} note first that the sequence $\{\bu^\epsilon\}_{\epsilon>0}$ converges in $L^2(D,\mathbb{R}^2)$ to $\bu^0$ and $\bu^0$ is in $SBD^2(D)$.  On passage to a subsequence if necessary it is seen that that $\{\bu^\epsilon\}_{\epsilon>0}$ converges almost everywhere to $\bu^0$.  Since the subsequence $u_1^\epsilon$ is even with respect to $x_2=0$ it is evident from \eqref{approx}   that $u^0_1$ is also even, a.e. with respect to two dimensional Lebesgue measure and from \eqref{approxjump} does not jump across the $x_2=0$ axis. Similarly since the subsequence $u^\epsilon_2$ is odd we find that $u_2^0$ is odd a.e. with respect to two dimensional Lebesgue measure. From \eqref{decomposedetailsS3} and lemma \ref{measures} the jump set $\mathcal{J}_{\bu^0}$ does not intersect $\{\ell^0 <x_1< a,\, x_2=0\}$. It now follows from \eqref{approxjump} and \eqref{reg}  that $u_2^0=0$ a.e. with respect to one dimensional $\mathcal{H}^1$ measure or equivalently Lebesgue measure on  $\{\ell^0 <x_1< a,\, x_2=0\}$ and the theorem is established.

\section{Convergence of nonlocal elastodynamics}\label{s:proofofconvergence}


In this section we give the proofs of lemmas \ref{convergingrhs}, \ref{convergences},  \ref{generaltractbdry}, \ref{generaltractoncrack},  and theorems \ref{limitspace} and \ref{momentumlim}. We begin with the derivation of theorem \ref{momentumlim}.
This is done with the aid of the following variational identities over properly chosen test spaces. The first variational identity over the domain $D$ is given in the following lemma. 
\begin{lemma}\label{momentumlim2}
For  a.e. $t\in (0,T)$ we have 
\begin{equation}\label{momentumlimit}
\begin{aligned}
\rho\langle\ddot\bu^0,\bw\rangle
=-\int_{D }\mathbb{C}\mathcal{E}\bu^0:\mathcal{E}\bw\,dx+\int_{\partial D}\,\bg\cdot\bw\, d\sigma,\hbox{ for all $\bw \in \dot H^{1}(D,\mathbb{R}^2)$},
\end{aligned}
\end{equation}
where $\langle\cdot,\cdot\rangle$ is the duality paring between $\dot H^{1}(D,\mathbb{R}^2)$ and its Hilbert space dual $\dot H^{1}(D,\mathbb{R}^2)'$. 
\end{lemma}

The next variational identity applies to the domains $L^\pm_\beta(t)$ adjacent to the moving crack.
\begin{lemma}\label{traction2}
The field $\ddot\bu_\tau^0(t)$ is a bounded linear functional on the spaces  \,$W^\pm(D_\beta(\tau))$ for a.e. $t\in (\tau,T)$ and we have 
\begin{equation}\label{momentumlimit2}
\begin{aligned}
\rho\langle\ddot\bu_\tau^0,\bw\rangle=-&\int_{L^\pm_\beta(\tau) }\mathbb{C}\mathcal{E}\bu^0:\mathcal{E}\bw\,\,d\bx+\int_{\partial D^\pm_\beta(\tau)}\,\bg\cdot\bw\, d\sigma,\\
& \hbox{ for all $\bw \in W^\pm(D_\beta(\tau))$}.
\end{aligned}
\end{equation}
\end{lemma}

We now prove theorem \ref{momentumlim} using lemmas \ref{generaltractbdry} and \ref{generaltractoncrack} and the variational identities given above by Lemmas \ref{momentumlim2} and \ref{traction2}. We may choose test functions $\bw$ in $H^1_0(D,\mathbb{R}^2)\subset \dot H^1(D,\mathbb{R}^2)$ in \eqref{momentumlimit} to see that
 \begin{equation}\label{momore}
\rho\ddot\bu_\tau^0=div\left(\mathbb{C}\mathcal{E}\bu^0\right) 
\end{equation}
as elements of $H^{-1}(D,\mathbb{R}^2)$ and \eqref{mo} of theorem  \ref{momentumlim} is established.
The traction on $\partial D$ given by \eqref{boundary tracesAB} now follows immediately from lemma \ref{generaltractbdry} and lemma \ref{momentumlim2}. Similarly the zero traction force acting on the component of $\partial D_{\beta}(\tau)^\pm$  lying on the crack faces  given by \eqref{boundary tracesG} now follows immediately from lemma \ref{generaltractoncrack}  and lemma \ref{traction2}. This concludes the proof of theorem \ref{momentumlim}.

Lemmas \ref{generaltractbdry} and \ref{generaltractoncrack} will be shown to follow from a generalized trace formula on the boundary of a Lipschitz domain $\Omega$. We call the domain $\Omega$ a polygon when it is a Lipschitz domain with smooth curvilinear arcs for edges $E_i$,  $i=1,\ldots M$, connected by vertices. We introduce the Sobolev space defined on $\Omega$ given by
\begin{equation}\label{cracklayertest}
H^{1,0}(\Omega,\mathbb{R}^2)=\left\{\bw\in H^1(\Omega,\mathbb{R}^2)\text{ and }\gamma \bw=0 \text{ on  a subset of edges}\right\},
\end{equation}
here $H^{1,0}(\Omega,\mathbb{R}^2)\subset \dot H^{1}(\Omega,\mathbb{R}^2)$.

\begin{lemma}\label{Gentracepologinal}
Given a domain $\Omega$ with Lipschitz boundary and let $\bu^0$ be an element of $SBV^2(\Omega)$, let $\bff$ be an element of $\dot H^1(\Omega,\mathbb{R}^2)'$, and
 \begin{equation}\label{momoremore}
div\left(\mathbb{C}\mathcal{E}\bu^0\right)=\bff
\end{equation}
as elements of $H^{-1}(\Omega,\mathbb{R}^2)$. 
Suppose first that test functions $\bw$ belong to $\dot H^1(\Omega,\mathbb{R}^2)$
and define $ \mathbb{C}\,\mathcal{E}\bu^0\bn$ on $\partial \Omega$ by
\begin{equation}\label{generaltwo}
\langle \mathbb{C}\,\mathcal{E}\bu^0\bn,\gamma\bw\rangle=\int_{\Omega}\mathbb{C}\,\mathcal{E}\bu^0:\mathcal{E}{\bw}\,d\bx+\langle\bff,\bw\rangle
\end{equation}
for all $\bw$ in $\dot H^1(\Omega,\mathbb{R}^2)$.
Then the functional $\langle \mathbb{C}\,\mathcal{E}\bu^0\bn,\gamma\bw\rangle$ is uniquely defined for all test functions $\bw$ in $\dot H^1(\Omega,\mathbb{R}^2)$, hence $\mathbb{C}\,\mathcal{E}\bu^0\bn$ belongs to $H^{-1/2}(\partial \Omega)$. 

Next suppose $\Omega$ is a polygon. Let $\bw$ belong to $H^{1,0}(\Omega,\mathbb{R}^2)$ and let $\bff$ be an element of $H^{1,0}(\Omega,\mathbb{R}^2)'$ and let $div\left(\mathbb{C}\mathcal{E}\bu^0\right)$ and $\bff$ satisfy \eqref{momoremore} as  elements of $H^{-1}(\Omega,\mathbb{R}^2)$.  Define $ \mathbb{C}\,\mathcal{E}\bu^0\bn$ on $\partial \Omega$ by
\begin{equation}\label{generaltwo2}
\langle \mathbb{C}\,\mathcal{E}\bu^0\bn,\gamma\bw\rangle=\int_{\Omega}\mathbb{C}\,\mathcal{E}\bu^0:\mathcal{E}{\bw}\,d\bx+\langle\bff,\bw\rangle
\end{equation}
for all $\bw$ in $H^{1,0}(\Omega,\mathbb{R}^2)$.
The functional $\langle \mathbb{C}\,\mathcal{E}\bu^0\bn,\gamma\bw\rangle$ is uniquely defined for all test functions $\bw$ in $H^{1,0}(\Omega,\mathbb{R}^2)$, hence $\mathbb{C}\,\mathcal{E}\bu^0\bn$ belongs to the dual space $H_{00}^{-1/2}(\partial \Omega)$. 

\end{lemma}

We now prove lemmas \ref{generaltractbdry} and \ref{generaltractoncrack}. With the hypothesis of  lemma  \ref{generaltractbdry}  we apply lemma \ref{momentumlim2} with test functions $\bw$ in $H^1_0(D,\mathbb{R}^2)\subset \dot H^1(D,\mathbb{R}^2)$ in \eqref{momentumlimit} to see as before
 \begin{equation}\label{momoremore2}
\rho\ddot\bu^0=div\left(\mathbb{C}\mathcal{E}\bu^0\right),
\end{equation}
as elements of $H^{-1}(\Omega,\mathbb{R}^2)$. Then we set $\bff=\rho\ddot\bu^0$ and lemma \ref{generaltractbdry} follows immediately from the first part of lemma \ref{Gentracepologinal}.
Now we see that the domains $L^{\pm}_\beta(t)$ of lemma \ref{generaltractoncrack} are polygons. With the hypothesis of  lemma \ref{generaltractoncrack}  we apply lemma \ref{traction2} and first consider test functions $\bw$ in $W^\pm(D_{\beta}(\tau))$ that vanish on the boundary of $L^\pm_\beta(t)$. Substitution into  \eqref{momentumlimit2}  gives
 \begin{equation}\label{momoremore3}
\rho\ddot\bu_\tau^0=div\left(\mathbb{C}\mathcal{E}\bu^0\right),
\end{equation}
as elements of $H^{-1}(L^\pm_\beta(t),\mathbb{R}^2)$. 
Note that $\bw\in W^\pm(D_\beta(t))$ implies that the restriction of $\bw$ to $L^\pm_\beta(t)$ belongs to
\begin{equation}\label{cracklayernormalspace}
H^{1,0}(L^\pm_\beta(t),\mathbb{R}^2)=\left\{\bw\in H^1(L^\pm_\beta(t),\mathbb{R}^2)\text{ and }\gamma \bw =0 \text{ on }  \partial L^\pm \right\},
\end{equation}
so we set we set $\bff=\rho\ddot\bu^0$ and lemma \ref{generaltractoncrack} follows immediately from the second part of lemma \ref{Gentracepologinal}.

We now prove the lemmas introduced in this section. We  begin with the proof of lemma \ref{Gentracepologinal} following  \cite{McLean}. To fix ideas we prove the second part of lemma \ref{Gentracepologinal} noting the first part follows identical lines. First note if $\bu^0$ belongs to $SBD^2(\Omega)$ then $\int_{\Omega}\,\mathbb{C}\mathcal{E}\bu^0:\mathcal{E}\bw\,d\bx$ as a map from $\bw\in H^{1,0}(\Omega,\mathbb{R}^2)$ to $\mathbb{R}$ belongs to $H^{1,0}(\Omega,\mathbb{R}^2)'$. Second note that the trace operator mapping $H^{1,0}(\Omega,\mathbb{R}^2)$ to $H_{00}^{-1/2}(\Omega)$ has a continuous right inverse denoted by $\tau$. We define $\tilde\bg$ by
\begin{equation}\label{generaltwo4}
\langle \tilde\bg,\bv\rangle=\int_{\Omega}\mathbb{C}\,\mathcal{E}\bu^0:\mathcal{E}{\tau\bv}\,d\bx+\langle\bff,\tau\bv\rangle
\end{equation}
for all $\bv$ in $H^{-1/2}_{00}(\partial \Omega)$ to show
\begin{equation}\label{generaltwo5}
\langle \tilde\bg,\gamma\bw\rangle=\int_{\Omega}\mathbb{C}\,\mathcal{E}\bu^0:\mathcal{E}{\bw}\,d\bx+\langle\bff,\bw\rangle
\end{equation}
for all $\bw$ in $H^{1,0}(\Omega,\mathbb{R}^2)$. 
To see this pick $\bw$ in $H^{1,0}(\Omega,\mathbb{R}^2)$ and set $\bw_0=\bw-\tau\gamma\bw$ so  $\bw_0$ is in $H^1_0(\Omega,\mathbb{R}^2)$ and from \eqref{momoremore} we have
\begin{equation}\label{eqlib}
\begin{aligned}
-\int_{\Omega}\,\mathbb{C}\mathcal{E}\bu^0:\mathcal{E}\bw_0\,d\bx=\langle\bw_0,\bff \rangle,
\end{aligned}
\end{equation}
so
\begin{equation}\label{eqlib6}
\begin{aligned}
-\int_{\Omega}\,\mathbb{C}\mathcal{E}\bu^0:\mathcal{E}\bw\,d\bx+\int_{\Omega}\,\mathbb{C}\mathcal{E}\bu^0:\mathcal{E}\tau\gamma\bw\,d\bx=\langle\bw,\bff \rangle-\langle\tau\gamma\bw,\bff \rangle.
\end{aligned}
\end{equation}
Equation \eqref{generaltwo5} follows directly from \eqref{eqlib6}, \eqref{generaltwo4}, and manipulation.
Now we show that the definition of $\tilde\bg$ given by \eqref{generaltwo4} is unique and independent of the choice of right inverse (lift) $\tau$.
Suppose we have $\bg^*$ defined by the lift $\tau^*$ given by
\begin{equation}\label{generaltwo7}
\langle \bg^*,\bv\rangle=\int_{\Omega}\mathbb{C}\,\mathcal{E}\bu^0:\mathcal{E}{\tau^*\bv}\,d\bx+\langle\bff,\tau^*\bv\rangle
\end{equation}
for all $\bv$ in $H^{-1/2}_{00}(\partial \Omega)$. From \eqref{generaltwo5} and linearity we get
\begin{equation}\label{generaltwo8}
\langle \tilde\bg-\bg^*,\gamma\bw\rangle=0,
\end{equation}
for all $\bw$ in $H^{1,0}(\Omega,\mathbb{R}^2)$ and uniqueness follows. We define $\mathbb{C}\mathcal{E}\bu^0\bn=\tilde\bg$ and the second part of lemma \ref{Gentracepologinal} is proved.

Next we give the proof of lemma \ref{convergingrhs}. First we show that the sequence $\{\bb^{\epsilon_n}(t)\}$ is uniformly bounded in $H^1(D,\mathbb{R}^2)'$ for $t\in [0,T]$. Let $\chi^{\epsilon_n}=\chi_+^{\epsilon_n}+\chi_-^{\epsilon_n}$ where $\chi^{\epsilon_n}_\pm$ are the indicator functions of the body force layers defined in \eqref{layers} so recalling \eqref{l2} we have for any $\bw \in H^1(D,\mathbb{R}^2)$,
\begin{equation}\label{newident}
\begin{aligned}
&\int_{D}\bb^{\epsilon_n}(\bx,t)\cdot\bw(\bx)\,d\bx=\int_{D}\frac{1}{\epsilon_n}\chi^{\epsilon_n}(\bx)\bg(x_1,t)\cdot\bw(\bx)\,d\bx\\
&=\int_{D}\frac{1}{\sqrt{\epsilon_n}}\chi^{\epsilon_n}(\bx)\bg(x_1,t)\cdot\frac{1}{\sqrt{\epsilon_n}}\chi^{\epsilon_n}(\bx)\bw(\bx)\,d\bx\\
&\leq \left(\int_{D}\frac{1}{\epsilon_n}\chi^{\epsilon_n}|\bg(t)|^2\,d\bx\right)^{1/2}\left(\int_{D}\frac{1}{\epsilon_n}\chi^{\epsilon_n}(\bx)|\bw|^2\,d\bx\right)^{1/2}\\
&\leq 2\Vert g_+(t)\Vert_{L^2(\theta,a-\theta)}I_{\epsilon_n}.
\end{aligned}
\end{equation}
Here $I_{\epsilon_n}$ is given by
\begin{equation}\label{newident22}
\begin{aligned}
I_{\epsilon_{n}}&=\left(\int_{D}\frac{1}{\epsilon_n}\chi^{\epsilon_n}(\bx)|\bw|^2\,d\bx\right)^{1/2}\\
&=\left(\int_0^1\int_\theta^{a-\theta}|\bw(x_1,\frac{b}{2}+\epsilon_n(y_2-1))|^2dx_1dy_2\right.\\
&+\left. \int_0^1\int_\theta^{a-\theta}|\bw(x_1,-\frac{b}{2}+\epsilon_n(1-y_2)|^2dx_1dy_2\right)^{1/2}
\end{aligned}
\end{equation}
where the change of variables $x_2=\pm\frac{b}{2}\mp\epsilon_n\pm\epsilon_n y_2$ has been made. From the change of variable it is evident that the factor $I_{\epsilon_n}$ is bounded above by
\begin{equation}\label{newident2}
\begin{aligned}
I_{\epsilon_n}\leq\left(\int_0^1\int_{\partial D_{\delta(y)}}|\bw|^2\,ds\,dy\right)^{1/2}
\end{aligned}
\end{equation}
where $D_{\delta(y)}=\{\bx\in D:\,\text{dist}(\bx,\partial D)>{\delta(y)}\}$ and $\delta (y)=\epsilon_n  (1-y)$, for $0<y<1$.
Since the trace operator is a bounded linear transformation between $H^1(D_{\delta(y)},\mathbb{R}^2)$ and $L^2(\partial D_{\delta(y)})^2$ we have
\begin{equation}\label{newident3}
\begin{aligned}
\int_{\partial D_{\delta(y)}}|\bw|^2\,ds\leq C_{\delta(y)}\Vert\bw\Vert^2_{H^1(D_{\delta(y)},\mathbb{R}^2)}\leq C_{\delta(y)}\Vert\bw\Vert^2_{H^1(D,\mathbb{R}^2)}.
\end{aligned}
\end{equation}
Additionally $C_{\delta(y)}$ depends only on the Lipschitz constant of the boundary \cite{Evans2} so for the case at hand we see that
\begin{equation}\label{Lip}
\begin{aligned}
\sup_{y\in[0,1]}\{C_{\delta(y)}\}<\infty,
\end{aligned}
\end{equation}
and from \eqref{newident}, \eqref{newident2}, and \eqref{Lip}  we conclude that there is a constant $C$ independent of $t$ and $\epsilon_n$ such that
\begin{equation}\label{newident4}
\begin{aligned}
|\int_{D}\bb^{\epsilon_n}(\bx,t)\cdot\bw(\bx)\,d\bx|\leq  C\Vert\bw\Vert^2_{H^1(D,\mathbb{R}^2)},
\end{aligned}
\end{equation}
so
\begin{equation}\label{bounded}
\begin{aligned}
\sup_{\epsilon_n>0}\int_0^T\Vert \bb^{\epsilon_n}(t)\Vert^2_{H^1(D;\mathbb{R}^2)'}dt<\infty.
\end{aligned}
\end{equation}
Thus we can pass to a subsequence also denoted by $\{\bb^{\epsilon_n}\}_{n=1}^\infty$ that converges weakly to $\bb^0$ in $L^2(0,T;H^1(D;\mathbb{R}^2)')$. Next we identify the weak limit $\bb^0(t)$ for a dense set of trial fields. Let $\bw\in C^1(\overline{D},\mathbb{R}^2)$ then a change of variables $x_2=\pm\frac{b}{2}\mp\epsilon_n\pm\epsilon_n y_2$ gives
\begin{equation}\label{newident8}
\begin{aligned}
\int_{D}\bb^{\epsilon_n}(\bx,t)\cdot\bw(\bx)\,d\bx&=\int_{D}\frac{1}{\epsilon_n}\chi^{\epsilon_n}(\bx)\bg(x_1,t)\cdot\bw(\bx)\,d\bx\\
&=\int_0^1\int_\theta^{a-\theta}\,g_+(x_1,t)\be^2\cdot\bw(x_1,\frac{b}{2}+\epsilon_n(y_2-1)\,dx_1\,dy_2\\
&+\int_0^1\int_\theta^{a-\theta}\,g_-(x_1,t)\be^2\cdot\bw(x_1,-\frac{b}{2}+\epsilon_n(1-y_2)\,dx_1\,dy_2.
\end{aligned}
\end{equation}
One passes to the $\epsilon_n\rightarrow 0$ limit in \eqref{newident8} applying the uniform continuity of $\bw$ to obtain
\begin{equation}\label{newident10}
\begin{aligned}
\lim_{\epsilon_n\rightarrow 0}\int_{D}\bb^{\epsilon_n}(\bx,t)\cdot\bw(\bx)\,d\bx=\int_{\partial D}\bg\cdot\bw\,d\sigma.
\end{aligned}
\end{equation}
Lemma \ref{convergingrhs} now follows noting that $C^1(\overline{D},\mathbb{R}^2)$ is dense in $H^1(D,\mathbb{R}^2)$.

We now establish lemma \ref{convergences}. The strong convergence 
\begin{equation}\label{convgences1}
\begin{aligned}
\bu^{\epsilon_n}\rightarrow \bu^0 & \hbox{  \rm strong in } C([0,T];\dot L^2(D;\mathbb{R}^2))
\end{aligned}
\end{equation}
follows immediately from the same arguments used to establish
theorem 5.1 of \cite{CMPer-Lipton}.
The weak convergence
\begin{equation}\label{convgences2}
\begin{aligned}
\dot\bu^{\epsilon_n}\rightharpoonup \dot\bu^0 &\hbox{ \rm weakly in  }L^2(0,T;\dot L^2(D;\mathbb{R}^2))\end{aligned}
\end{equation}
follows noting that theorem 2.2 of \cite{CMPer-Lipton} shows that
\begin{equation}
\begin{aligned}
\sup_{\epsilon_n>0}\int_0^T\Vert \dot\bu^{\epsilon_n}(t)\Vert^2_{L^2(D;\mathbb{R}^2)}dt<\infty.
\end{aligned}
\end{equation}
Thus we can pass to a subsequence also denoted by $\{\dot\bu^{\epsilon_n}\}_{n=1}^\infty$ that converges weakly to $\dot\bu^0$ in $L^2(0,T;\dot L^2(D;\mathbb{R}^2))$. 

To prove
\begin{equation}\label{convgences3}
\begin{aligned}
\ddot\bu^{\epsilon_n}\rightharpoonup \ddot\bu^0 &\hbox{ \rm weakly in }L^2(0,T;\dot H^1(D;\mathbb{R}^2)')
\end{aligned}
\end{equation}
we must show that
\begin{equation}\label{blf}
\sup_{\epsilon_n>0}\int_0^T\Vert\ddot\bu^{\epsilon_n}(t)\Vert^2_{\dot H^1(D;\mathbb{R}^2)'}\,dt<\infty,
\end{equation}
and existence of a weakly converging sequence follows.
We multiply \eqref{energy based model2} with a test function $\bw$ from $\dot H^1(D;\mathbb{R}^2)$ and integrate over $D$.  

A straightforward integration  by parts gives 
\begin{equation}\label{energy based weakform est1}
\begin{aligned}
&\int_D\ddot{\bu}^{\epsilon_n}(\bx,t)\cdot {\bw}(\bx)d\bx\\
&=-\frac{1}{\rho}
\int_D \int_{H_{\epsilon_n}(\bx)\cap D} |\by-\bx|\partial_S\mathcal{W}^{\epsilon_n}(S(\by,\bx,\bu^{\epsilon_n}(t)))S(\by,\bx,\bw)\,d\by d\bx\\
&+\frac{1}{\rho}\int_D\bb^{\epsilon_n}(\bx,t)\cdot\bw(\bx)d\bx,
\end{aligned}
\end{equation}
and we now estimate the right hand side of \eqref{energy based weakform est1}.
The first term on the righthand side is denoted by $I^{\epsilon_n}$ and we change variables $\by=\bx+\epsilon\xi$, $|\xi|<1$, with $d\by=\epsilon_n^2d\xi$  and write out $\partial_S\mathcal{W}^\epsilon(S(\by,\bx,\bu^\epsilon(t)))$ to get
\begin{equation}\label{est1}
\begin{aligned}
&I^{\epsilon_n}=-\frac{1}{\rho\omega_2}\int_{{D}\times{\mathcal{H}_1(0)}}\omega(\bx,,\epsilon_n\xi)|\xi|J(|\xi|)h'\left(\epsilon_n|\xi||D_{\be}^{\epsilon_n|\xi|}\bu^{\epsilon_n}\cdot \be|^2\right)\\
&\times2(D_{\be}^{\epsilon_n|\xi|}\bu^{\epsilon_n}\cdot \be) (D_{\be}^{\epsilon_n|\xi|}\bw\cdot \be)\,d\xi\,d\bx,
\end{aligned}
\end{equation}
where $\omega(\bx,\epsilon_n\xi)$ is unity if $\bx+\epsilon_n\xi$ is in $D$ and zero otherwise. We define the sets
\begin{equation}\label{overstress}
\begin{aligned}
A^-_{\epsilon_n}&=\left\{(\bx,\xi) \hbox{ in }D\times\mathcal{H}_1(0);\,|D_{\be}^{\epsilon_n|\xi|}\bu^{\epsilon_n}\cdot\be|<\frac{{r}^c}{\sqrt{\epsilon_n|\xi|}}\right\}\\
A^+_{\epsilon_n}&=\left\{(\bx,\xi) \hbox{ in }D\times\mathcal{H}_1(0);\,|D_{\be}^{\epsilon_n|\xi|}\bu^{\epsilon_n}\cdot\be|\geq\frac{{r}^c}{\sqrt{\epsilon_n|\xi|}}\right\},
\end{aligned}
\end{equation}
with $D\times\mathcal{H}_1(0)=A^-_{\epsilon_n}\cup A^+_{\epsilon_n}$ and we write
\begin{equation}\label{i[plus}
I^{\epsilon_n} =I_1^{\epsilon_n}+I_2^{\epsilon_n},
\end{equation}
where
\begin{equation}\label{est2}
\begin{aligned}
I_1^{\epsilon_n}=-\frac{1}{\rho\omega_2}\int_{D\times\mathcal{H}_1(0)\cap A_{\epsilon_n}^-}\omega(\bx,,\epsilon_n\xi)|\xi|J(|\xi|)h'\left(\epsilon_n|\xi||D_{\be}^{\epsilon_n|\xi|}\bu^{\epsilon_n}\cdot \be|^2\right)\\
\times 2(D_{\be}^{\epsilon_n|\xi|}\bu^{\epsilon_n}\cdot \be) (D_{\be}^{\epsilon_n|\xi|}\bw\cdot \be)\,d\xi\,d\bx,\\
I_2^{\epsilon_n}=-\frac{1}{\rho\omega_2}\int_{D\times\mathcal{H}_1(0)\cap A_{\epsilon_n}^+}\omega(\bx,,\epsilon_n\xi)|\xi|J(|\xi|)h'\left(\epsilon_n|\xi||D_{\be}^{\epsilon_n|\xi|}\bu^{\epsilon_n}\cdot \be|^2\right)\\
\times2(D_{\be}^{\epsilon_n|\xi|}\bu^{\epsilon_n}\cdot \be) (D_{\be}^{\epsilon_n|\xi|}\bw\cdot \be)\,d\xi\,d\bx,
\end{aligned}
\end{equation}
In what follows we will denote positive constants independent of $\bu^{\epsilon_n}$ and $\bw\in \dot H^1(D;\mathbb{R}^2)$ by $C$. 
First note that $h$ is concave so $h'(r)$ is monotone decreasing for $r\geq 0$ and from Cauchy's inequality, and \eqref{L2bound} one has
\begin{equation}\label{est3}
\begin{aligned}
|I_1^{\epsilon_n}|&\leq\frac{2h'(0)C}{\rho\omega_2}\left(\int_{D\times\mathcal{H}_1(0)\cap A_{\epsilon_n}^-}\omega(\bx,,\epsilon_n\xi)|D_{\be}^{\epsilon_n|\xi|}\bw\cdot \be)|^2\,d\xi\,d\bx\right)^{1/2},\\
&\leq \frac{2h'(0)C}{\rho\omega_2}\left(\int_{\mathcal{H}_1(0)}\int_{D}\omega(\bx,,\epsilon_n\xi)|D_{\be}^{\epsilon_n|\xi|}\bw\cdot \be)|^2\,d\bx\,d\xi\right)^{1/2},
\end{aligned}
\end{equation}
Since $\bx$ and $\bx+\epsilon_n\xi$ belong to $D$ we write $\xi=|\xi|\be$ where $\be=\xi/|\xi|$ and calculation gives
\begin{equation}\label{diffquo2}
 D_{\be}^{\epsilon_n|\xi|}\bw\cdot \be=\int_0^1\mathcal{E}\bw(\bx+s\epsilon_n|\xi|\be)\be\cdot\be\,ds,
\end{equation}
with $\bx+s\epsilon_n|\xi|\be\in D$ for $0<s<1$. Next introduce $\chi_D(\bx+s\epsilon_n|\xi|\be)$ taking the value $1$, if $\bx+s\epsilon_n|\xi|\be\in D$ and $0$ otherwise.  Substitution of \eqref{diffquo2} into \eqref{est3}  and application of the Jensen inequality and Fubini's theorem gives
\begin{equation}\label{fourth est}
\begin{aligned}
|I_1^{\epsilon_n}|&\leq\frac{2h'(0)C}{\rho\omega_2}\left(\int_0^1\int_{\mathcal{H}_1(0)}\int_{D}\chi_D(\bx+s\epsilon_n|\xi|\be)|\mathcal{E}\bw(\bx+s\epsilon_n|\xi|\be)\be\cdot\be|^2\,d\bx\,d\xi\,ds\right)^{1/2},
\end{aligned}
\end{equation}
and we conclude
\begin{equation}\label{est4}
\begin{aligned}
|I_1^{\epsilon_n}|&\leq C\Vert \bw\Vert_{H^1({D};\mathbb{R}^2)}.
\end{aligned}
\end{equation}
Elementary calculation gives the estimate  (see equation (6.53) of \cite{CMPer-Lipton})
\begin{equation}\label{max}
\sup_{0\leq x<\infty}|h'({\epsilon_n}|\xi|{x}^2)2x|\leq\frac{2h'(\overline{r}^2){\overline{r}}}{\sqrt{\epsilon_n|\xi|}},
\end{equation}
and we also have (see equation (6.78) of \cite{CMPer-Lipton})
\begin{equation}
\int_{D\times\mathcal{H}_1(0)\cap A^{+}_{\epsilon_n}} \omega(\bx,\xi)J(|\xi|)\,d\xi\,d\bx<C\epsilon_n,
\label{Fprimesecond}
\end{equation}
so  Cauchy's inequality and the inequalities \eqref{diffquo2}, \eqref{max}, \eqref{Fprimesecond}  give
\begin{equation}\label{est5}
\begin{aligned}
|I_2^{\epsilon_n}|&\leq\frac{1}{\rho\omega_2}\int_{D\times\mathcal{H}_1(0)\cap A_{\epsilon_n}^+}\omega(\bx,,\epsilon_n\xi)|\xi|J(|\xi|)\frac{2h'(\overline{r}^2){\overline{r}}}{\sqrt{\epsilon_n|\xi|}}|D_{\be}^{\epsilon_n|\xi|}\bw\cdot \be|\,d\xi\,d\bx,\\
&\leq\frac{1}{\rho\omega_2}\left(\int_{D\times\mathcal{H}_1(0)\cap A_{\epsilon_n}^+}\omega(\bx,,\epsilon_n\xi)|\xi|J(|\xi|)\frac{(2h'(\overline{r}^2){\overline{r}})^2}{{\epsilon_n|\xi|}}\,d\xi\,d\bx\right)^{1/2}\times\\
&\left(\int_{D\times\mathcal{H}_1(0)\cap A_{\epsilon_n}^+}\omega(\bx,\epsilon_n\xi)|\xi|J(|\xi|)|D_{\be}^{\epsilon_n|\xi|}\bw\cdot \be|^2\,d\xi\,d\bx\,dt\right)^{1/2}\\
&\leq C\Vert \bw\Vert_{H^1({D};\mathbb{R}^2)},
\end{aligned}
\end{equation}
and we conclude that the first term on the right hand side of \eqref{energy based weakform est1} admits the estimate
\begin{equation}\label{fin}
\begin{aligned}
|I^{\epsilon_n}| &\leq |I_1^{\epsilon_n}|+|I_2^{\epsilon_n}| \leq  C\Vert \bw\Vert_{H^1({D};\mathbb{R}^2)},
\end{aligned}
\end{equation}
for all $\bw\in H^1(D;\mathbb{R}^2)$.

It follows immediately from lemma \ref{convergingrhs} that the second term on the right hand side of  \eqref{energy based weakform est1} satisfies the estimate
\begin{equation}\label{blf2}
\begin{aligned}
\frac{1}{\rho}\left|\int_D\bb^{\epsilon_n}(\bx,t)\cdot\bw(\bx)\,d\bx\right |\leq C\Vert\bw\Vert_{H^1(D;\mathbb{R}^2)},\hbox{ for all $\bw\in H^1(D;\mathbb{R}^2)$}
\end{aligned}
\end{equation}
From \eqref{fin} and \eqref{blf2} we conclude that there exists a $C>0$ so that
\begin{eqnarray}\label{energy based weakform est3}
\begin{aligned}
\left |\int_D\ddot{\bu}^{\epsilon_n}(\bx,t)\cdot {\bw}(\bx)\,d\bx\right |\leq C\Vert\bw\Vert_{H^1(D;\mathbb{R}^2)},\hbox{ for all $\bw\in \dot H^1(D;\mathbb{R}^2)$}
\end{aligned}
\end{eqnarray}
so
\begin{equation}\label{bounded}
\begin{aligned}
&\sup_{\epsilon_n>0}\sup_{t\in [0,T]}\frac{\int_D\ddot{\bu}^{\epsilon_n}(\bx,t)\cdot {\bw}(\bx)d\bx}{\Vert\bw\Vert_{H^1(D;\mathbb{R}^2)}}<C,\hbox{ for all $\bw\in \dot H^1(D;\mathbb{R}^2)$},
\end{aligned}
\end{equation}
or
\begin{equation}\label{linfinitybounded}
\begin{aligned}
\sup_{t\in [0,T]}\Vert\ddot{\bu}^{\epsilon_n}(t)\Vert_{H^1(D;\mathbb{R}^2)'}<C, \hbox{ for all $\epsilon_n$}
\end{aligned}
\end{equation}
and \eqref{blf} follows. The estimate \eqref{blf} implies weak compactness and passing to subsequences if necessary we deduce that
$\ddot\bu^{\epsilon_n}\rightharpoonup \ddot\bu^0 $ weakly in $L^2(0,T;\dot H^1(D;\mathbb{R}^2)')$ and lemma \ref{convergences} is proved.

To establish lemma \ref{momentumlim2}
we take a test function $\varphi(t)\bw(\bx)$ with $\varphi\in C^\infty_c(0,T)$ and $\bw$ in $C^\infty(\overline{D},\mathbb{R}^2)$ orthogonal to rigid body motions. 
Substituting this test function into \eqref{energy based weakform} and integration by parts in time gives 
\begin{equation}\label{energy based weakform est15}
\begin{aligned}
&\int_0^T\varphi(t)\rho\int_D\ddot{\bu}^{\epsilon_n}(\bx,t)\cdot {\bw}(\bx)d\bx\,dt\\
 &=-\int_0^T\varphi(t)\int_D \int_{\mathcal{H}_{\epsilon_n}(\bx)\cap D} |\by-\bx|\partial_S\mathcal{W}^{\epsilon_n}(S(\by,\bx,\bu^{\epsilon_n}(t)))S(\by,\bx\,\bw)\,d\by \,d\bx\,dt\\
&+\int_0^T\varphi(t)\int_D\bb^{\epsilon_n}(\bx,t)\cdot\bw(\bx)\,d\bx\,dt,
\end{aligned}
\end{equation}
The goal is to pass to the $\epsilon_n=0$  limit in this equation to recover \eqref{momentumlimit}.
The limit of the left hand side of \eqref{energy based weakform est15} follows from  Lemma \ref{convergences} 
\begin{equation}\label{momentumlimitfirstterm}
\begin{aligned}
\lim_{\epsilon_n\rightarrow 0}\int_0^T\varphi(t)\rho\int_D\ddot{\bu}^{\epsilon_n}(\bx,t)\cdot {\bw}(\bx)d\bx dt&=\int_0^T\varphi(t)\,\rho\langle \ddot\bu^0(t),\bw\rangle\,dt.
\end{aligned}
\end{equation}
To recover the $\epsilon_n=0$ limit of the first term on the right hand side of  \eqref{energy based weakform est15} we appeal to the bound \eqref{fin} to pass to the limit under the time integral using Lebesgue dominated convergence. 
Next apply Lemma 6.5 of \cite{CMPer-Lipton} with straightforward modifications to get
\begin{equation}\label{1st termrightside decomp}
\begin{aligned}
&\lim_{\epsilon_n\rightarrow 0}I^{\epsilon_n}=-\lim_{\epsilon_n\rightarrow 0}\int_D \int_{\mathcal{H}_{\epsilon_n}(\bx)\cap D} |\by-\bx|\partial_S\mathcal{W}^{\epsilon_n}(S(\by,\bx,\bu^{\epsilon_n}(t)))S(\by,\bx,\bw)\,d\by d\bx\\
&=-\lim_{\epsilon_n\rightarrow 0}\frac{2}{\omega_2}\int_{{D}\times{\mathcal{H}_1(0)}}\omega(\bx,\epsilon_n\xi)|\xi|J(|\xi|)h'(0)
(D_{\be}^{\epsilon_n|\xi|}\bu^{\epsilon_n}\cdot \be)^- (D_{\be}^{\epsilon_n|\xi|}\bw\cdot \be)\,d\xi\,d\bx,
\end{aligned}
\end{equation}
where $S(\by,\bx,\bu^\epsilon(t))^-=(D_{\be}^{\epsilon|\xi|}\bu^\epsilon\cdot\be)^-$.
As indicated in section \ref{s:proofofsymmetry}, $D_{\be}^{{\epsilon_n}|\xi|}\bu^{\epsilon_n}\cdot \be^-\rightharpoonup g(\bx,\xi)$ converges weakly in $L^2(D\times\mathcal{H}_1(0),\mathbb{R}^2)$ with respect to the measure $|\xi|J(|\xi|)d\xi d\bx $ and $D_{\be}^{\epsilon_n|\xi|}\bw\cdot \be\rightarrow\mathcal{E}\,\bw\,\be\cdot\be$ uniformly on $D$, so
\begin{equation}
\label{1st termrightside decomp1}
\begin{aligned}
&\lim_{\epsilon_n\rightarrow 0}I^{\epsilon_n}=\frac{2}{\omega_2}\int_{{D}\times{\mathcal{H}_1(0)}}\omega(\bx,\epsilon_n\xi)|\xi|J(|\xi|)h'(0)
g(\bx,\xi)\mathcal{E}\,\bw\,\be\cdot\be \,d\xi\,d\bx.,\\
\end{aligned}
\end{equation}
and from \eqref{weaklim} $g(\bx,\xi)=\mathcal{E}\,\bu^0\,\be\cdot\be$ and we recover
\begin{equation}\label{1st termrightside decomp2}
\begin{aligned}
&\lim_{\epsilon_n\rightarrow 0}I^{\epsilon_n}=-\int_{D}\mathbb{C}\mathcal{E}\bu^0:\mathcal{E}\bw\,\,d\bx,
\end{aligned}
\end{equation}
so
\begin{equation}\label{1st termrightside decomp2final}
\begin{aligned}
&\lim_{\epsilon_n\rightarrow 0}\int_0^T\varphi(t)\,I^{\epsilon_n}dt=-\int_0^T\varphi(t)\int_{D}\mathbb{C}\mathcal{E}\bu^0:\mathcal{E}\bw\,\,d\bx\,dt.
\end{aligned}
\end{equation}
We pass to the limit in the second term on the right hand side of  \eqref{energy based weakform est15} using lemma \ref{convergingrhs} to obtain 
\begin{equation}\label{momentumlimitalmost}
\begin{aligned}
\int_0^T\varphi(t)\rho\langle\ddot\bu^0(t),\bw\rangle\,dt
=-\int_0^T\varphi(t)\left(\int_{D }\mathbb{C}\mathcal{E}\bu^0:\mathcal{E}\bw\,dx+\int_{\partial D}\,\bg\cdot\bw\, d\sigma\right)\,dt.
\end{aligned}
\end{equation}
From the density of $C^\infty(\overline{D},\mathbb{R}^2)$ in $\bw\in\dot H^{1}(D,\mathbb{R}^2)$  we see that \eqref{momentumlimitalmost} holds for all $\bw\in\dot H^{1}(D,\mathbb{R}^2)$.  Since  \eqref{momentumlimitalmost} holds for all $\varphi\in C^\infty_c(0,T)$  we recover \eqref{momentumlimit}.

We now establish theorem \ref{limitspace} to show that  $\ddot\bu_\tau^0(\bx,t)$ is a bounded linear functional on the spaces  \,$W^\pm(D_{\beta}(\tau))$ for a.e. $t\in (\tau,T)$.  We illustrate the proof for $\bw\in W^+(D_\beta(\tau))$ noting that identical steps hold for $\bw\in W^-(D_\beta(\tau))$. Pick $\tau\in (0,T)$, suppose $\tau<t$, multiply \eqref{energy based model2} by a trial $\bw\in W^+(D_\beta(\tau))$ and integrating by parts over $D$ to gives
\begin{equation}\label{energy based weakform est 200}
\begin{aligned}
&\rho\int_D\ddot{\bu}^{\epsilon_n}(\bx,t)\cdot {\bw}(\bx)\,d\bx\\
 &=-
\int_{D} \int_{H_{\epsilon_n}(\bx)\cap D} |\by-\bx|\partial_S\mathcal{W}^{\epsilon_n}(S(\by,\bx,\bu^{\epsilon_n}(t)))S(\by,\bx,\bw)\,d\by\, d\bx
\\
&+\int_{D}\bb^{\epsilon_n}(\bx,t)\cdot\bw(\bx)\,d\bx
\end{aligned}
\end{equation}
Now we show that $\ddot\bu^{\epsilon_n}(t)$ is bounded in $W^+(D_\beta(\tau))'$ uniformly for all $t\in (\tau,T)$ and $0<\epsilon_n<\beta/2$.
As before the first term on the righthand side is denoted by $I^{\epsilon_n}$ and we change variables $\by=\bx+\epsilon\xi$, $|\xi|<1$, with $d\by=\epsilon_n^2d\xi$  and write out $\partial_S\mathcal{W}^\epsilon(S(\by,\bx,\bu^\epsilon(t)))$ to get
\begin{equation}\label{est98}
\begin{aligned}
&I^{\epsilon_n}=-\frac{1}{\omega_2}\int_{{D}\times{\mathcal{H}_1(0)}}\omega(\bx,\epsilon_n\xi)|\xi|J(|\xi|)h'\left(\epsilon_n|\xi||D_{\be}^{\epsilon_n|\xi|}\bu^{\epsilon_n}\cdot \be|^2\right)\\
&\times2(D_{\be}^{\epsilon_n|\xi|}\bu^{\epsilon_n}\cdot \be) (D_{\be}^{\epsilon_n|\xi|}\bw\cdot \be)\,d\xi\,d\bx,
\end{aligned}
\end{equation}
where $\omega(\bx,\xi)$ is unity if $\bx+\epsilon_n\xi$ is in $D$ and zero otherwise. Note that the boundary component  of $\partial D^+_\beta(\tau)$ given by $\{\bx\in D:\,\ell(0)\leq x_1\leq
\ell^0(\tau)-\beta,\,x_2=0\}$ is a subset of the failure zone centerline $C^{\epsilon_n}(t)$ so for $\bx$ and $\by$ in $FZ^{\epsilon_n}(t)$ we see that ${\bolds{f}}^{\epsilon_n}(\by,\bx)=0$ or equivelently
\begin{equation}\label{ident998}
\begin{aligned}
h'\left(\epsilon_n|\xi||D_{\be}^{\epsilon_n|\xi|}\bu^{\epsilon_n}\cdot \be|^2\right)\times2(D_{\be}^{\epsilon_n|\xi|}\bu^{\epsilon_n}\cdot \be)=0
\end{aligned}
\end{equation}
for $\bx$ and $\by$ in $FZ^{\epsilon_n}(t)$.
Then for for $n$ large enough so that $\ell^0(\tau)-\beta<\ell^{\epsilon_n}(t)$ and $0<\epsilon_n<\beta/2$ and for test functions $\bw\in W^+(D_\beta(\tau))$ the product
\begin{equation}\label{ident98}
\begin{aligned}
&h'\left(\epsilon_n|\xi||D_{\be}^{\epsilon_n|\xi|}\bu^{\epsilon_n}\cdot \be|^2\right)\times2(D_{\be}^{\epsilon_n|\xi|}\bu^{\epsilon_n}\cdot \be) (D_{\be}^{\epsilon_n|\xi|}\bw\cdot \be)\\
&=h'\left(\epsilon_n|\xi||D_{\be}^{\epsilon_n|\xi|}\bu^{\epsilon_n}\cdot \be^-|^2\right)\times2(D_{\be}^{\epsilon_n|\xi|}\bu^{\epsilon_n}\cdot \be^-) (D_{\be}^{\epsilon_n|\xi|}\bw\cdot \be)\\
&=\chi(\bx,\bx+\epsilon_n\xi)h'\left(\epsilon_n|\xi||D_{\be}^{\epsilon_n|\xi|}\bu^{\epsilon_n}\cdot \be^-|^2\right)\times2(D_{\be}^{\epsilon_n|\xi|}\bu^{\epsilon_n}\cdot \be^-) (D_{\be}^{\epsilon_n|\xi|}\bw\cdot \be),
\end{aligned}
\end{equation}
where
\begin{align}\label{infervalfunction}
\chi(\bx,\bx+\epsilon_n\xi)= \begin{cases}
0, & \text{ if the points } \bx,\,\bx+\epsilon_n\xi \text{ are separated by } \{0\leq x_1\leq \ell^0(\tau)-\beta,\, x_2=0\}\\
1,& \text{otherwise} .
\end{cases}
\end{align}
(Here we say that $\bx,\,\bx+\epsilon_n\xi$ are separated by $\{0\leq x_1\leq \ell^0(\tau)-\beta,\, x_2=0\}$ when it is impossible to connect these two points by a line segment without crossing $\{0\leq x_1\leq \ell^0(\tau)-\beta,\, x_2=0\}$.)
Then $I^{\epsilon_n}$ becomes
\begin{equation}\label{est99}
\begin{aligned}
&I^{\epsilon_n}=-\frac{1}{\omega_2}\int_{{D}\times{\mathcal{H}_1(0)}}\omega(\bx,\epsilon_n\xi)\chi(\bx,\bx+\epsilon_n\xi)|\xi|J(|\xi|)h'\left(\epsilon_n|\xi||D_{\be}^{\epsilon_n|\xi|}\bu^{\epsilon_n}\cdot \be^-|^2\right)\\
&\times2(D_{\be}^{\epsilon_n|\xi|}\bu^{\epsilon_n}\cdot \be^-) (D_{\be}^{\epsilon_n|\xi|}\bw\cdot \be)\,d\xi\,d\bx.
\end{aligned}
\end{equation}
We can now bound \eqref{est99} as in \eqref{est3} and change the order of integration to arrive at the upper bound
\begin{equation}\label{fourth estt}
\begin{aligned}
|I^{\epsilon_n}|&\leq\frac{2h'(0)C}{\omega_2}\left(\int_{\mathcal{H}_1(0)}\int_{D}\omega(\bx,\epsilon_n\xi)\chi(\bx,\bx+\epsilon_n\xi)| D_{\be}^{\epsilon_n|\xi|}\bw\cdot \be  |^2\,d\bx\,d\xi\,\right)^{1/2}.
\end{aligned}
\end{equation}
We change to slicing variables and write $\bx=\by+r \be$, where $\be$ is on the unit circle and $\by\in\Pi_{\be}$ where $\Pi_{\be}$ is the subspace perpendicular to $\be$ and $r\in\mathbb{R}$. 
We set $D^{\be}_{\by}=\{r\in\mathbb{R}:\,\by+r\be\in D_\beta(\tau)\}$  and $D^{\be}=\{\by\in\Pi_{\be}:\, D^{\be}_{\by}\not=\emptyset\}$ so
\begin{equation}\label{fourth esttt}
\begin{aligned}
|I^{\epsilon_n}|&\leq\frac{2h'(0)C}{\omega_2}\left(\int_{\mathcal{H}_1(0)}\int_{D^{\be}}\int_{D^{\be}_{\by}}\chi(\by+r\be,\by+(r+\epsilon_n|\xi|)\be)| D_{\be}^{\epsilon_n|\xi|}\bw\cdot \be  |^2\,dr\,d\by\,d\xi\,\right)^{1/2}.
\end{aligned}
\end{equation}
We use the fact that functions in Sobolev spaces are absolutely continuous for a.e. lines to write \eqref{diffquo2} for $\bw\in W^+(D_\beta(\tau))$ and
\begin{equation}\label{fourth estttt}
\begin{aligned}
&|I^{\epsilon_n}|\leq\\
&\frac{2h'(0)C}{\omega_2}\left(\int_{\mathcal{H}_1(0)}\int_{D^{\be}}\int_{D^{\be}_{\by}}\chi(\by+r\be,\by+(r+\epsilon_n|\xi|)\be)| \int_0^1\mathcal{E}(\by+(r+s\epsilon_n|\xi|\be))\be\cdot \be \,ds |^2\,dr\,d\by\,d\xi\,\right)^{1/2}\\
&\leq\\
&\frac{2h'(0)C}{\omega_2}\left(\int_{\mathcal{H}_1(0)} \int_0^1\int_{D^{\be}}\int_{D^{\be}_{\by}}\chi(\by+r\be,\by+(r+\epsilon_n|\xi|)\be)|\mathcal{E}(\by+(r+s\epsilon_n|\xi|\be))\be\cdot \be |^2\,dr\,d\by\,ds\,d\xi\,\right)^{1/2}.
\end{aligned}
\end{equation}
where Jensen inequality and Fubini's theorem have been applied in the last line. Introducing  $\chi_{D_\beta(\tau)}(\bx)=1$ if its argument lies in $D_\beta(\tau)$ and zero otherwise, applying $\chi(\by+r\be,\by+(r+\epsilon_n|\xi|)\be)\leq\chi_{D_\beta(\tau)}(\by+r\be)\chi_{D_\beta(\tau)}(\by+(r+s\epsilon_n|\xi|)\be)$ and changing to original variables gives
\begin{equation}\label{fourth est}
\begin{aligned}
|I^{\epsilon_n}|&\leq\frac{2h'(0)C}{\omega_2}\left(\int_0^1\int_{\mathcal{H}_1(0)}\int_{D}\chi_{D_\beta(\tau)}(\bx)\chi_{D_\beta(\tau)}(\bx+s\epsilon_n|\xi|\be)|\mathcal{E}\bw(\bx+s\epsilon_n|\xi|\be)\be\cdot\be|^2\,d\bx\,d\xi\,ds\right)^{1/2}.
\end{aligned}
\end{equation}
From this we conclude
\begin{equation}\label{est40}
\begin{aligned}
|I^{\epsilon_n}|&\leq C\Vert \bw\Vert_{H^1({D_\beta(\tau)};\mathbb{R}^2)}.
\end{aligned}
\end{equation}
Arguments identical to the proof of lemma \ref{convergingrhs} show that the sequence $\bb^{\epsilon_n}$ is uniformly bounded in $W^+(D_\beta(\tau))'$ for all $\tau\in [0,T]$ and $\epsilon_n>0$ and together with \eqref{est40} one concludes
\begin{equation}\label{linfinitybounded20}
\begin{aligned}
\sup_{t\in (\tau,T)}\Vert\ddot{\bu}^{\epsilon_n}(t)\Vert_{W^+(D_\beta(\tau);\mathbb{R}^2)'}<C, \hbox{ for $\beta/2>\epsilon_n>0$}.
\end{aligned}
\end{equation}
Hence
\begin{equation}\label{linfinitybounded200}
\begin{aligned}
\int_\tau^T\Vert\ddot{\bu}^{\epsilon_n}(t)\Vert_{W^+(D_\beta(\tau);\mathbb{R}^2)'}^2\,dt<\infty \hbox{ for $\beta/2>\epsilon_n>0$},
\end{aligned}
\end{equation}
and passing to a subsequence if necessary  gives a $\bv(t)$ in $L^2(\tau,T;\,W^+(D_\beta(\tau))')$ such that $\ddot\bu^{\epsilon_n}\rightharpoonup\bv$ weakly in $L^2(\tau,T;\,W^+(D_\beta(\tau)')$.

We finish the proof by showing $\bv=\ddot\bu^0_\tau$. To see this note $\bu^{\epsilon_n}\in C^2([0,T];\,L^2(D,\mathbb{R}^2))$ and for $\varphi\in C_c^\infty(\tau,T)$ and for $\bw\in W^+(D_\beta(\tau))$ we have
\begin{equation}\label{weakidentity}
\int_\tau^T\int_{D}\ddot\bu^{\epsilon_n}\cdot\bw\,d\bx\,\varphi(t)\,dt=-\int_\tau^T\int_{D}\dot\bu^{\epsilon_n}\cdot\bw\,d\bx\,\dot\varphi(t)\,dt
\end{equation}
Passing to the $\epsilon_n=0$ limit using lemma \ref{convergences} applied to the right hand side gives
\begin{equation}\label{weakidentityy}
\int_\tau^T\langle\bv,\bw\rangle\,\varphi(t)dt=-\int_\tau^T\int_{D}\dot\bu^0\cdot\bw\,\dot\varphi(t) d\bx\,dt, \hbox{ for all $\bw\in W^+(D_\beta(\tau))$}
\end{equation}
and we deduce from \eqref{weakidentityy} that $\bv=\ddot\bu^0_\tau$ as elements of $W^+(D_\beta(\tau))'$.
Identical arguments show that $\ddot\bu^0_\tau\in W^-(D_\beta(\tau))'$ and Theorem \ref{limitspace} is proved.

We now prove lemma \ref{traction2}.
We illustrate the proof for $\bw(\bx)\in W^+(D_\beta(\tau))$ noting an identical proof holds for $\bw \in W^-(D_\beta(\tau))$.
Pick $\varphi(t)\in C^\infty_c(\tau,T)$ and  $\bw(\bx)\in W^+(D_\beta(\tau))$ and substitute into  \eqref{energy based weakform} and an integration by parts in time gives
\begin{equation}\label{energy based weakform est 200}
\begin{aligned}
&\int_\tau^T\rho\int_D\ddot{\bu}^{\epsilon_n}(\bx,t)\cdot {\bw}(\bx)\,d\bx\,\varphi(t)\,dt\\
 &=-
\int_\tau^T\int_{D} \int_{H_{\epsilon_n}(\bx)\cap D} |\by-\bx|\partial_S\mathcal{W}^{\epsilon_n}(S(\by,\bx,\bu^{\epsilon_n}(t)))S(\by,\bx,\bw)\,d\by\, d\bx\,\varphi(t)\,dt
\\
&+\int_\tau^T\int_{D}\bb^{\epsilon_n}(\bx,t)\cdot\bw(\bx)\,d\bx\,\varphi(t)\,dt
\end{aligned}
\end{equation}
The goal is to pass to the $\epsilon_n=0$  limit in this equation to recover \eqref{momentumlimit2}. The limit of the left hand side of \eqref{energy based weakform est 200} follows from  Theorem \ref{limitspace}
\begin{equation}\label{momentumlimitfirstterm}
\begin{aligned}
\lim_{\epsilon_n\rightarrow 0}\int_\tau^T\varphi(t)\rho\int_D\ddot{\bu}^{\epsilon_n}(\bx,t)\cdot {\bw}(\bx)d\bx dt&=\int_\tau^T\varphi(t)\,\rho\langle \ddot\bu^0_\tau(t),\bw\rangle\,dt.
\end{aligned}
\end{equation}

The first term on the right hand side of \eqref{energy based weakform est 200} is written
\begin{equation}\label{1sttermrhs}
\int_\tau^T\varphi I^{\epsilon_n}\,dt.
\end{equation}
We can recover the $\epsilon_n=0$ limit of the first term on the right hand side of  \eqref{energy based weakform est 200} by appealing to the bound \eqref{est40} to pass to the limit under the time integral using Lebesgue dominated convergence once we show that for every $\bw\in W^+(D_\beta(\tau))$ the bounded sequence $\{I^{\epsilon_n}(t)\}$ has a limit for a.e. $t\in(\tau,T)$. To see this we apply \eqref{ident98} to get that
\begin{equation}\label{weak strong}
\begin{aligned}
&I^{\epsilon_n}(t)=-\int_D \int_{\mathcal{H}_{\epsilon_n}(\bx)\cap D} |\by-\bx|\partial_S\mathcal{W}^{\epsilon_n}(S(\by,\bx,\bu^{\epsilon_n}(t)))S(\by,\bx,\bw)\,d\by d\bx\\
&=-\frac{1}{\omega_2}\int_{{D}\times{\mathcal{H}_1(0)}}\omega(\bx,\epsilon_n\xi)\chi(\bx,\bx+\epsilon_n\xi)|\xi|J(|\xi|)h'\left(\epsilon_n|\xi||D_{\be}^{\epsilon_n|\xi|}\bu^{\epsilon_n}\cdot \be^-|^2\right)\\
&\times2(D_{\be}^{\epsilon_n|\xi|}\bu^{\epsilon_n}\cdot \be^-) (D_{\be}^{\epsilon_n|\xi|}\bw\cdot \be)\,d\xi\,d\bx.
\end{aligned}
\end{equation}
The integrand is the product of two factors (note $\omega(\bx,\epsilon_n\xi)\chi(\bx,\bx+\epsilon_n\xi)=\omega(\bx,\epsilon_n\xi)^2\chi(\bx,\bx+\epsilon_n\xi)^2$) and we show that on passing to a subsequence if necessary the first factor
\begin{equation}\label{weak}
\omega(\bx,\epsilon_n\xi)\chi(\bx,\bx+\epsilon_n\xi)h'\left(\epsilon_n|\xi||D_{\be}^{\epsilon_n|\xi|}\bu^{\epsilon_n}\cdot \be^-|^2\right)
\times2(D_{\be}^{\epsilon_n|\xi|}\bu^{\epsilon_n}\cdot \be^-)\rightharpoonup 2h'(0)g(\bx,\xi,t))
\end{equation}
weakly in $L^2(D\times\mathcal{H}_1(0),\mathbb{R})$ and the second factor
\begin{equation}\label{strong}
\omega(\bx,\epsilon_n\xi)\chi(\bx,\bx+\epsilon_n\xi)D_{\be}^{\epsilon_n|\xi|}\bw\cdot \be\rightarrow\mathcal{E}\bw(\bx)\be\cdot\be.
\end{equation}
strong in $L^2(D\times\mathcal{H}_1(0),\mathbb{R})$.
Here as in section \ref{s:proofofsymmetry} the $L^2$ norm and inner product are with respect to the weighted measure $|\xi|J(|\xi|)d\xi d\bx $.
Hence for fixed $t$ we conclude that for any cluster point of $\{I^{\epsilon_n}(t)\}$  there is a subsequence
\begin{equation}\label{limit}
\begin{aligned}
&\lim_{\epsilon_{n'}\rightarrow 0}I^{\epsilon_{n'}}(t)=-\int_D\int_{\mathcal{H}_1(0)}\,|\xi|J(|\xi|)2h'(0)g(\bx,\xi,t))\mathcal{E}\bw(\bx)\be\cdot\be\,d\xi\, d\bx\\
&=-\int_D\int_{\mathcal{H}_1(0)}\,2|\xi|J(|\xi|)h'(0)(\mathcal{E}\bu^0(t,\bx)\be\cdot\be)(\mathcal{E}\bw(\bx)\be\cdot\be)\,d\xi\, d\bx \\
&=-\int_D\mathbb{C}\mathcal{E}\bu^0(t,\bx):\mathcal{E}\bw(\bx)\,d\bx,
\end{aligned}
\end{equation}
where the second line  follows from \eqref{weaklim} and the third line follows from a calculation.
One obtains the same limit for subsequences of all possibly distinct cluster points of $\{I^{\epsilon_n}(t)\}$ to conclude there is one cluster point and we have identified $\lim_{\epsilon_n\rightarrow 0}I^{\epsilon_n}(t)$ for a.e. $t\in(0,T)$.

To conclude the weak and strong convergences \eqref{weak} and \eqref{strong} are established. First note that $h'(r)$ is monotone decreasing in $r$ so  $h'\left(\epsilon_n|\xi||D_{\be}^{\epsilon_n|\xi|}\bu^{\epsilon_n}\cdot \be^-|^2\right)\leq h'(0)$ and from \eqref{L2bound} $D_{\be}^{\epsilon_n|\xi|}\bu^{\epsilon_n}\cdot\be^-$ is bounded in $L^2(D\times\mathcal{H}_1(0),\mathbb{R})$ so the first factor is bounded in  $L^2(D\times\mathcal{H}_1(0),\mathbb{R})$ uniformly in $\epsilon_n$ and has a subsequence that converges weakly to a limit written $K(\bx,\xi,t)$. Application of Lemma  6.5 of \cite{CMPer-Lipton} and \eqref{weakdiffquo} allows us to identify  $K(\bx,\xi,t)=2h'(0)g(\bx,\xi,t))$ where we have explicitly written the time dependence of $g(\bx,\xi)$ and weak convergence is established.  Next we form
\begin{equation}\label{strongconvg}
\begin{aligned}
&A^{\epsilon_n}=\frac{1}{\omega_2}\int_{{D}\times{\mathcal{H}_1(0)}}\omega(\bx,\epsilon_n\xi)\chi(\bx,\bx+\epsilon_n\xi)|\xi|J(|\xi|)\left |(D_{\be}^{\epsilon_n|\xi|}\bw\cdot \be)-\mathcal{E}\bw(\bx)\be\cdot\be\right |^2\,d\xi\,d\bx.
\end{aligned}
\end{equation}
Procceding as before we get
\begin{equation}\label{strongconvgence}
\begin{aligned}
&\lim_{\epsilon_n\rightarrow 0}A^{\epsilon_n}\leq\\
&\lim_{\epsilon_n\rightarrow0}\int_0^1\frac{1}{\omega_2}\int_{{D}\times{\mathcal{H}_1(0)}}\chi_{D_\beta(\tau)}(\bx)\chi_{D_\beta(\tau)}(\bx+s\epsilon_n\xi)|\xi|J(|\xi|)\left|\mathcal{E}\bw(\bx+s\epsilon_n|\xi|\be)-\mathcal{E}\bw(\bx)\be\cdot\be\right|^2\,d\xi\,d\bx\,ds\\
&=\int_0^1\int_{D_\beta(\tau)}s^2\lim_{\epsilon_n\rightarrow0}\frac{1}{s^2\omega_2}\int_{\mathcal{H}_1(0)}\chi_{D_\beta(\tau)}(\bx+s\epsilon_n\xi)|\xi|J(|\xi|)\left|\mathcal{E}\bw(\bx+s\epsilon_n|\xi|\be)-\mathcal{E}\bw(\bx)\be\cdot\be\right|^2\,d\xi\,d\bx\,ds\\
&=0,
\end{aligned}
\end{equation}
where we use Lebesgue bounded convergence to interchange limit and integral and the point wise limit holds a.e. $\bx\in D_\beta(\tau)$ at the Lebesgue points
\begin{equation}\label{pointwiseconvgence}
\begin{aligned}
\lim_{\epsilon_n\rightarrow0}\frac{1}{s^2\omega_2}\int_{\mathcal{H}_1(0)}|\xi|J(|\xi|)\left|\mathcal{E}\bw(\bx+s\epsilon_n|\xi|\be)-\mathcal{E}\bw(\bx)\be\cdot\be\right|^2\,d\xi=0,
\end{aligned}
\end{equation}
This establishes strong convergence for $\bw\in W^+(D_\beta(\tau))$.
Collecting results gives that the limit of the first term on the right hand side of \eqref{energy based weakform est 200} is
\begin{equation}\label{1st termrightside decomp3final}
\begin{aligned}
&\lim_{\epsilon_n\rightarrow 0}\int_0^T\varphi(t)\,I^{\epsilon_n}dt=-\int_0^T\varphi(t)\int_{D}\mathbb{C}\mathcal{E}\bu^0:\mathcal{E}\bw\,\,d\bx\,dt.
\end{aligned}
\end{equation}
Passing to the limit on the last term of the right hand side of  \eqref{energy based weakform est 200} and arguments similar to before give
\begin{equation}\label{1st termrightside decomp3finalfinal}
\begin{aligned}
&\lim_{\epsilon_n\rightarrow 0}\int_\tau^T\int_D\bb^{\epsilon_n}\cdot\bw\,d\bx\,\varphi(t)\,dt=\int_\tau^T\int_{\partial D}\bg\cdot\bw\,d\sigma\,\varphi(t)\,dt.
\end{aligned}
\end{equation}
and we conclude that
\begin{equation}\label{momentumlimitalmost2}
\begin{aligned}
\int_\tau^T\varphi(t)\rho\langle\ddot\bu^0(t),\bw\rangle\,dt
=-\int_\tau^T\varphi(t)\left(\int_{D }\mathbb{C}\mathcal{E}\bu^0:\mathcal{E}\bw\,dx+\int_{\partial D}\,\bg\cdot\bw\, d\sigma\right)\,dt,
\end{aligned}
\end{equation}
for all $\bw\in W^+(D_\beta(\tau))$ and lemma \ref{traction2} is proved.

\section{Weak solution of the wave equation on $D_t$}
\label{s:weaksolnproof}
Theorem \ref{zerohorizonweaksoln} is proved in this section.  From theorem \ref{symmetry2} and  lemma \ref{convergences}  the limit displacement $\bu^0$ belongs to $\mathcal{V}$. From lemma 2.8 and remark 2.9 of \cite{DalToader} we have that if $\bu\in\mathcal{V}$ and  \eqref{weaksoln} holds for every $\varphi \in C_c^\infty((0,T);V_T)$ with $\varphi(t)\in V_t$ then $\bu$ is a weak solution of \eqref{formulation}.  Motivated by this we begin by selecting a class of trial fields that are convenient to work with.  For $t \in [0,T]$ take $s_\beta(t)=t-\beta$ and given $\bw\in  C_c^\infty((0,T);V_T) $ with $\bw\in V_t$ for $t\in(0,T)$. Set $\tilde{\bw}(t)=\bw(s_\beta(t))\in V_{s_\beta(t)}\subset V_t$  for some $\beta\in (0,t)$.  Substitution of this trial in  \eqref{energy based weakform} gives the identity 
\begin{eqnarray}\label{energy based weakform2}
\begin{aligned}
&\rho \int_0^T\int_D\dot{\bu}^{\epsilon_n}(t)\cdot {\dot{\tilde\bw}}(t)d\bx\,dt\\
 &=
\int_0^T\int_D \int_{\mathcal{H}_{\epsilon_n}(\bx)\cap D} |\by-\bx|\partial_S\mathcal{W}^{\epsilon_n}(S(\by,\bx,\bu^{\epsilon_n}(t)))S(\by,\bx,\tilde\bw(t))\,d\by d\bx\,dt\\
&-\int_0^T\int_D\bb^{\epsilon_n}(t)\cdot\tilde\bw(t)\,d\bx\,dt, \hbox{   for $\epsilon_n>0$}.
\end{aligned}
\end{eqnarray}
Here we will pass to the $\epsilon_n=0$ limit in this identity to obtain an $\epsilon_n=0$ identity. Then on passing to the $\beta\rightarrow 0$  limit in each term we  will show that $\bu^0$ is a weak solution. We begin by understanding the limit of the middle term in \eqref{energy based weakform2}  for a given sequence indexed by $\epsilon_n$.  We write out the integrand appearing under the time integral 
\begin{equation}\label{limitofmiddleterm}
I^{\epsilon_n}(t,\tilde\bw(t))=\int_D \int_{\mathcal{H}_{\epsilon_n}(\bx)\cap D} |\by-\bx|\partial_S\mathcal{W}^{\epsilon_n}(S(\by,\bx,\bu^{\epsilon_n}(t)))S(\by,\bx,\tilde\bw(t))\,d\by\, d\bx.
\end{equation}
and identify the point-wise limit $\lim_{\epsilon_n\rightarrow 0} I^{\epsilon_n}(t,\tilde\bw(t))$ for a.e. $t\in(0,T)$. 
For this choice of test function we change variables as in \eqref{est1} to obtain
\begin{equation}\label{i[plus2}
I^{\epsilon_n}(t,\tilde\bw) =I_1^{\epsilon_n}(t,\tilde\bw)+I_2^{\epsilon_n}(t,\tilde\bw),
\end{equation}
where
\begin{equation}\label{est222222}
\begin{aligned}
I_1^{\epsilon_n}(t,\tilde\bw)=-\frac{1}{\omega_2}\int_{D\times\mathcal{H}_1(0)}\omega(\bx,,\epsilon_n\xi)|\xi|J(|\xi|)h'\left(\epsilon_n|\xi||D_{\be}^{\epsilon_n|\xi|}\bu^{\epsilon_n}\cdot \be^-|^2\right)\\
\times 2(D_{\be}^{\epsilon_n|\xi|}\bu^{\epsilon_n}\cdot \be^-) (D_{\be}^{\epsilon_n|\xi|}\tilde\bw\cdot \be)\,d\xi\,d\bx,\\
I_2^{\epsilon_n}(t,\tilde\bw)=-\frac{1}{\omega_2}\int_{D\times\mathcal{H}_1(0)\cap \{SZ^{\epsilon_n}(t)\setminus FZ^{\epsilon_n}(t)\}\cap A_{\epsilon_n}^+}\omega(\bx,,\epsilon_n\xi)|\xi|J(|\xi|)h'\left(\epsilon_n|\xi||D_{\be}^{\epsilon_n|\xi|}\bu^{\epsilon_n}\cdot \be|^2\right)\\
\times2(D_{\be}^{\epsilon_n|\xi|}\bu^{\epsilon_n}\cdot \be) (D_{\be}^{\epsilon_n|\xi|}\tilde\bw\cdot \be)\,d\xi\,d\bx,
\end{aligned}
\end{equation}
As in \eqref{est99} we have
\begin{equation}\label{est9999}
\begin{aligned}
&I_1^{\epsilon_n}(t,\tilde\bw)=-\frac{1}{\omega_2}\int_{{D}\times{\mathcal{H}_1(0)}}\omega(\bx,\epsilon_n\xi)\tilde\chi(\bx,\bx+\epsilon_n\xi)|\xi|J(|\xi|)h'\left(\epsilon_n|\xi||D_{\be}^{\epsilon_n|\xi|}\bu^{\epsilon_n}\cdot \be^-|^2\right)\\
&\times2(D_{\be}^{\epsilon_n|\xi|}\bu^{\epsilon_n}\cdot \be^-) (D_{\be}^{\epsilon_n|\xi|}\tilde\bw\cdot \be)\,d\xi\,d\bx,
\end{aligned}
\end{equation}
where
\begin{align}\label{infervalfunction2}
\tilde\chi(\bx,\bx+\epsilon_n\xi)= \begin{cases}
0, & \text{ if the points } \bx,\,\bx+\epsilon_n\xi \text{ are separated by } \{0\leq x_1\leq \ell^0(t-\beta),\, x_2=0\}\\
1,& \text{otherwise},
\end{cases}
\end{align}
for $n$ large enough so that $\ell^0(\beta-t)<\ell^{\epsilon_n}(t)$ and $0<\epsilon_n<(\ell^0(t)-\ell^0(t-\beta))/2$, where $\beta\in(0,t)$. (Here we have used that $\ell^0(t)$ is continuous and strictly increasing.)
As before the integrand is the product of two factors such that the first factor
\begin{equation}\label{weak}
\omega(\bx,\epsilon_n\xi)\tilde\chi(\bx,\bx+\epsilon_n\xi)h'\left(\epsilon_n|\xi||D_{\be}^{\epsilon_n|\xi|}\bu^{\epsilon_n}\cdot \be^-|^2\right)
\times2(D_{\be}^{\epsilon_n|\xi|}\bu^{\epsilon_n}\cdot \be^-)\rightharpoonup 2h'(0)g(\bx,\xi,t))
\end{equation}
weakly in $L^2(D\times\mathcal{H}_1(0),\mathbb{R})$ and the second factor
\begin{equation}\label{strong}
\omega(\bx,\epsilon_n\xi)\tilde\chi(\bx,\bx+\epsilon_n\xi)D_{\be}^{\epsilon_n|\xi|}\tilde\bw\cdot \be\rightarrow\mathcal{E}\tilde\bw(\bx)\be\cdot\be.
\end{equation}
strong in $L^2(D\times\mathcal{H}_1(0),\mathbb{R})$. Hence we conclude using the same arguments given in the proof of lemma \ref{traction2} that
\begin{equation}\label{I1limit}
\lim_{\epsilon_{n}\rightarrow 0}I_1^{\epsilon_{n}}(t,\tilde\bw)=-\int_D\mathbb{C}\mathcal{E}\bu^0(t):\mathcal{E}\tilde\bw\,d\bx.
\end{equation}
For  $\tilde\bw\in V_{s_\beta(t)}$ and hypothesis \ref{hyp1} and since $\ell^0(t)$ is strictly increasing and continuous it is evident that for  $\epsilon_n$ sufficiently small,  $\tilde\bw$ is continuous on almost all lines that intersect $\{SZ^{\epsilon_n}\setminus FZ^{\epsilon_n}\}$. From hypothesis \ref{hyp1}, noting that $\ell^0(t)$  is strictly increasing and continuous, we find after a simple calculation that $|\{SZ^{\epsilon_n}\setminus FZ^{\epsilon_n}\}|\leq C|\epsilon_n|^2$. We estimate $I_2^{\epsilon_n}(t,\tilde\bw)$ as in \eqref{est5}
\begin{equation}\label{est50000}
\begin{aligned}
|I_2^{\epsilon_n}(t,\tilde\bw)|&\leq\frac{1}{\rho\omega_2}\int_{D\times\mathcal{H}_1(0)\cap \{SZ^{\epsilon_n}(t)\setminus FZ^{\epsilon_n}(t)\}\cap A_{\epsilon_n}^+}\omega(\bx,,\epsilon_n\xi)|\xi|J(|\xi|)\frac{2h'(\overline{r}^2){\overline{r}}}{\sqrt{\epsilon_n|\xi|}}|D_{\be}^{\epsilon_n|\xi|}\tilde\bw\cdot \be|\,d\xi\,d\bx,\\
&\leq\frac{1}{\rho\omega_2}\left(\int_{D\times\mathcal{H}_1(0)\cap\{SZ^{\epsilon_n}(t)\setminus FZ^{\epsilon_n}(t)\}}\omega(\bx,,\epsilon_n\xi)|\xi|J(|\xi|)\frac{(2h'(\overline{r}^2){\overline{r}})^2}{{\epsilon_n|\xi|}}\,d\xi\,d\bx\right)^{1/2}\times\\
&\left(\int_{D\times\mathcal{H}_1(0)\cap \{SZ^{\epsilon_n}(t)\setminus FZ^{\epsilon_n}(t)\}}\omega(\bx,\epsilon_n\xi)|\xi|J(|\xi|)|D_{\be}^{\epsilon_n|\xi|}\tilde\bw\cdot \be|^2\,d\xi\,d\bx\,dt\right)^{1/2}\\
&\leq C\sqrt{|\epsilon_n|}\Vert \tilde\bw\Vert_{H^1({D_\beta(t)};\mathbb{R}^2)}.
\end{aligned}
\end{equation}
From this we conclude that $\lim_{\epsilon_n\rightarrow 0} I^{\epsilon_n}(t,\tilde\bw)$ exists and 
\begin{equation}\label{Ilimit}
\lim_{\epsilon_{n}\rightarrow 0}I^{\epsilon_{n}}(t,\tilde\bw)=-\int_D\mathbb{C}\mathcal{E}\bu^0(t):\mathcal{E}\tilde\bw\,d\bx,
\end{equation}
for $\tilde\bw \in V_{s_\beta(t)}$ for a.e. $t\in (0,T)$.  
Arguments identical to the proof of theorem  \ref{limitspace} show that for $\tilde\bw\in V_{s_\beta(t)}$ we have
\begin{equation}\label{est40000}
\begin{aligned}
|I^{\epsilon_n}(t,\tilde\bw)|&\leq C\Vert \tilde\bw\Vert_{V_{s_\beta(t)}}.
\end{aligned}
\end{equation}

We form
\begin{eqnarray}\label{energy based weakform middleterm}
\int_0^TI^{\epsilon_n}(t,{\tilde\bw}(t))\,dt.
\end{eqnarray}
One then sees from definition \ref{defweakspaces} that $\Vert{\tilde\bw}(t)\Vert_{V_{s_\beta(t)}}$ is integrable and from  \eqref{est40000} we can apply the Lebesgue dominated convergence theorem to conclude
\begin{eqnarray}\label{energy based weakform middletermlim}
\lim_{\epsilon_n\rightarrow 0}\int_0^TI^{\epsilon_n}(t,{\tilde\bw}(t))\,dt=-\int_0^T\int_D\mathbb{C}\mathcal{E}\bu^0(t):\mathcal{E}{\tilde\bw}(t)\,d\bx\,dt.
\end{eqnarray}

It is first noted that lemma \ref{convergingrhs} can be extended in a straight forward way to the present context. Applying this to the last term in \eqref{energy based weakform2}  gives
\begin{equation}\label{tractionbdryconddd}
-\lim_{\epsilon_n\rightarrow0}\int_0^T\int_D\bb^{\epsilon_n}(t)\cdot\tilde\bw(t)\,d\bx\,dt=-\int_0^T\int_{\partial D}\bg(t)\cdot\tilde\bw(t)\,d\sigma\,dt.
\end{equation}
We apply lemma \ref{convergences} to the first term of \eqref{energy based weakform2} and pass to a subsequence if necessary to find that
\begin{equation}\label{firstterm}
\lim_{\epsilon_n\rightarrow 0}\rho \int_0^T\int_D\dot{\bu}^{\epsilon_n}(t)\cdot \dot{\tilde\bw}(t)d\bx\,dt=\rho \int_0^T\int_D\dot{\bu}^{0}(t)\cdot \dot{\tilde\bw}(t)d\bx\,dt.
\end{equation}
On again passing to a subsequence if necessary we recover
\begin{equation}
\label{weaksolnbeta}
-\int_0^T\rho\int_D\dot\bu(t)\cdot\dot{\tilde\bw}(t)\,d\bx\,dt+\int_0^T\int_D\mathbb{C}\mathcal{E}\bu(t):\mathcal{E}\tilde\bw(t)\,d\bx\,dt=\int_0^T\int_{\partial D}\bg(t)\cdot\tilde\bw(t)\,d\sigma\,dt,
\end{equation}
where $\tilde\bw(t)=\bw(s_\beta(t))=\bw(t-\beta)\in V_{s_\beta(t)}$ for a.e. $t\in [0,T]$. Given that $\bw(t)\in C_c^\infty(0,T;V_T)$ we see that
\begin{equation}\label{firsttermsmooth}
\lim_{\beta\rightarrow 0}\rho \int_0^T\int_D\dot{\bu}^{0}(t)\cdot \dot{\bw}(t-\beta)d\bx\,dt=\rho \int_0^T\int_D\dot{\bu}^{0}(t)\cdot \dot{\bw}(t)d\bx\,dt.
\end{equation}
Similarly
\begin{equation}\label{tractionbdrycondbeta}
-\lim_{\beta\rightarrow0}\int_0^T\int_{\partial D}\bg^{0}(t)\cdot\tilde\bw(t)\,d\sigma\,dt=-\int_0^T\int_{\partial D}\bg(t)\cdot\bw(t)\,d\sigma\,dt.
\end{equation}
To finish the proof we show $\lim_{\beta\rightarrow 0}\bw(s_\beta(t))=\bw(t)$ in $V_t$, a.e. for $t\in [0,T]$.
We use the following lemma proved in \cite{DalLar}.
\begin{lemma}\label{density}
Let $\{V_t\}_{t\in[0,T]}$ be an increasing family of closed linear subspaces of a separable Hilbert space $V$. Then, there exists a countable set $S\subset [0,T]$ such that for all $t\in [0,T]\setminus S$, we have
\begin{equation}\label{closure}
V_t=\overline{\bigcup_{s<t}V_s}.
\end{equation}
\end{lemma}
Observe that
\begin{equation}\label{equevalence}
\bigcup_{0<\beta}V_{s_\beta(t)}=\bigcup_{s<t}V_s,
\end{equation}
so  $\lim_{\beta\rightarrow 0}\bw(s_\beta(t))=\bw(t)$ in $V_t$, a.e. for $t\in [0,T]$, hence
\begin{equation}\label{Ilimitbeta}
\lim_{\beta\rightarrow 0}\int_D\mathbb{C}\mathcal{E}\bu^0(t):\mathcal{E}\tilde\bw(t)\,d\bx=\int_D\mathbb{C}\mathcal{E}\bu^0(t):\mathcal{E}\bw(t)\,d\bx.
\end{equation}
Since $\bu^0\in\mathcal{V}$ it is also clear from Cauchy's inequality applied to \eqref{Ilimit} that for $\beta>0$ that
\begin{equation}\label{Ilimitbetabound}
\left|\int_D\mathbb{C}\mathcal{E}\bu^0(t):\mathcal{E}\tilde\bw(t)\,d\bx\right|\leq C\Vert\bw(t)\Vert_{V_T},
\end{equation}
and 
\begin{equation}\label{Ilimitbetatime}
\lim_{\beta\rightarrow 0}\int_0^T\int_D\mathbb{C}\mathcal{E}\bu^0(t):\mathcal{E}\tilde\bw(t)\,d\bx\,dt=\int_0^T\int_D\mathbb{C}\mathcal{E}\bu^0(t):\mathcal{E}\bw(t)\,d\bx\,dt.
\end{equation}
follows from the Lesbegue dominated convergence theorem. Collecting results we have
\begin{equation}
\label{weaksolnfinal}
-\int_0^T\rho\int_D\dot\bu(t)\cdot\dot{\bw}(t)\,d\bx\,dt+\int_0^T\int_D\mathbb{C}\mathcal{E}\bu(t):\mathcal{E}\bw(t)\,d\bx\,dt=\int_0^T\int_{\partial D}\bg(t)\cdot\bw(t)\,d\sigma\,dt,
\end{equation}
for all $\bw\in C_c^\infty((0,T);V_T)$ with $\bw(t)\in V_t$ and theorem \ref{zerohorizonweaksoln} is proved.


\section{Conclusions}
\label{s:conclusions}

In this paper we use a double well energy within a perydynamic formulation.  We provide the boundary value problem satisfied by the limit displacement $\bu^0$.  The limit displacement  $\bu^0(\bx,t)$  satisfies the boundary conditions of the dynamic brittle fracture problem given by
\begin{itemize}


\item  Prescribed inhomogeneous traction boundary conditions.

\item Balance of linear momentum as described by the linear elastic wave equation.

\item Zero traction on the sides of the evolving crack.

\item Displacement jumps can only occur inside the crack set $\Gamma_t$.

\end{itemize}
In this way the boundary value problem for the elastic field for dynamic Linear Elastic Fracture Mechanics (LEFM)  is recovered as described in \cite{Freund}, \cite{RaviChandar}, \cite{Anderson}, 
\cite{Slepian}. Moreover the limit displacement $\bu^0$ is a weak solution of the wave equation on the time dependent domain $D_t$ containing the running crack. This establishes a rigorous connection between the nonlocal fracture formulation using a peridynamic model derived from a double well potential and the wave equation posed on cracking domains given in \cite{DalToader}.

One can assume a more general crack structure for the nonlocal model and pass to the local limit to see that the nonlocal elastic displacements converge to limits that are weak solutions to the wave equation on a more general cracking domain. As an example we change body forces  and initial conditions as appropriate and consider a pice-wise smooth curve $\Gamma\subset D$ of length $L$ originating at the the notch $(x_1=\ell(0),\,x_2=0)$ containing all nonlocal crack centerlines for $t\in[0,T]$. For a given horizon the crack centerline at time $t$ is characterized by the curve $J^\epsilon(t)$ originating at  the notch $(x_1=\ell(0),\,x_2=0)$   of length $\sigma^\epsilon(t)$ at time $t$. The centerline length grows and is assumed to be an increasing function in time. The failure zone is defined as
\begin{equation}
\begin{aligned}
FZ^\epsilon(t)=\{\bx\in D,\,\xi\in\mathcal{H}_1(0):\, \bx=\by+\epsilon\xi, \,\hbox{ and }\by\in J^\epsilon(t)\}.
\end{aligned}
\label{failure2}
\end{equation}
We introduce the curve $\tilde{J}^\epsilon(t)$ lying on $\Gamma$ and containing $J^\epsilon(t)$ with length $\sigma^\epsilon+C\epsilon$
and the softening zone is defined by
\begin{equation}
\begin{aligned}
SZ^\epsilon(t)=\{\bx\in D,\,\xi\in\mathcal{H}_1(0):\, \bx=\by+\epsilon\xi,\,\hbox{ and }\by\in \tilde{J}^\epsilon(t)\},
\end{aligned}
\label{failure2}
\end{equation}
where $FZ^\epsilon(t)\subset SZ^\epsilon(t)$.
As before  we can pass to a subsequence $\epsilon_n\rightarrow0$ to find an increasing distance $\sigma^0(t)$. 
 Lemma \ref{measures} extends to this case and additionally when  $\sigma^0(t)$ is continuous and strictly increasing 
we apply arguments identical to those given in section \ref{s:weaksolnproof} to show that the limit displacement $\bu^0(t)$ is the weak solution to the wave equation inside the cracking domain. 
 More generally it is conjectured that nonlocal elastodynamics converge to weak solutions of the wave equation for growing cracks described by closed countably rectifiable subsets of $D$ with bounded one dimensional Hausdorff measure.


\newcommand{\noopsort}[1]{}


\begin{thebibliography}{34}
\expandafter\ifx\csname natexlab\endcsname\relax\def\natexlab#1{#1}\fi
\expandafter\ifx\csname url\endcsname\relax
  \def\url#1{\texttt{#1}}\fi
\expandafter\ifx\csname urlprefix\endcsname\relax\def\urlprefix{URL }\fi



 
 
 

 
 
 
 
 
\bibitem{AmbrosioCosicaDalmaso}
Ambrosio, L., Coscia, A., and Dal Maso, G., 1997. Fine properties of functions with bounded deformation. 
Arch. Ration. Mech. Anal. 139,  201--238.

\bibitem{Anderson}
 Anderson, T.~L. 2005. Fracture Mechanics: Fundamentals and Applications. 3rd edition. Taylor \& Francis, Boca Raton.
 
 \bibitem{Bobaru 2015}
Bobaru, F. and Zhang, G., 2015.  Why do cracks branch? A peridynamic investigation of dynamic brittle fracture. International Journal of Fracture {196}, 59--98.
 
 \bibitem{Bellido}
Bellido, J.~C., Morra-Corral, C., and Pedregal, P. 2015. Hyperelasticity as a $\Gamma$-limit of peridynamics when the horizon goes to zero
Calc. Var. DOI 10.1007/s00526-015-0839-9.

\bibitem{Hughes}
Borden, M., Verhoosel, C., Scott, M., Hughes, T., and Landis, C. 2012. {A phase-field description of dynamic brittle fracture.} Computer Methods in Applied Mechanics and Engineering {217-220} 77-95.

\bibitem{BourdinLarsenRichardson} 
Bourdin, B., Larsen, C., and Richardson, C. 2011. { A time-discrete model for dynamic fracture based on crack regularization.} Int. J. Fract.  {168}  133--143.

\bibitem{ContiFocardiIurlano}
Conti, S, Focardi, M., and Iurlano, F., 2018. Which special functions of bounded deformation have bounded variation?
Proceedings of the Royal Society of Edinburgh, Section A: Mathematics 148, 33--50.


\bibitem{DalLarsenToader}
Dal Maso G., Larsen C.J., Toader R., 2020. Elastodynamic Griffith fracture on prescribed crack paths with kinks. NoDEA Nonlinear Differential Equations Appl. 27, no 1. Art. 4.

\bibitem{DalLar}
Dal Maso, G., Larsen, C.J., 2011. Existence for wave equations on domains with arbitrary growing cracks. Atti 
Accad. Naz. Lincei Cl. Sci. Fis. Mat. Natur. Rend. Lincei (9) Mat. Appl. 22, 387--408.



\bibitem{DalToader}
Dal Maso, G. and Toader, R., 2019. On the Cauchy problem for the wave equation on time dependent domains. J. Differ. Equ. 266, 3209 -- 3246.

 \bibitem{erick}
 Ericksen, J.~L., 1975. Equilibrium of bars. J. Elasticity 5, 191--202.

\bibitem{Evans}
Evans, L.~C., 1998. Partial Differential Equations. American Mathematical Society. Providence, RI.

\bibitem{Evans2}
Evans, L.~C. and Gariepy, R.~F.,1992. Measure Theory and Fine Properties of Functions. CRC Press. Boca Raton, Ann Arbor, London.




\bibitem{Freund}
Freund, B. 1990. Dynamic Fracture Mechanics. Cambridge Monographs on Mechanics and Applied Mathematics. Cambridge University Press. Cambridge.

\bibitem{FreundClifton}
Freund, B. and Clifton, R~J., 1974. On the uniqueness of plane elastodynamic solutions for running cracks.
Journal of Elasticity 4, 293--299.









\bibitem{KO}
Kikuchi, N. and Oden, J. T., 1988. Contact Problems in Elasticity. SIAM, Philadelphia.





\bibitem{Jha-Lipton2020}
Jha, P.~K. and Lipton R., 2020. Kinetic relations and local energy balance for LEFM from a nonlocal peridynamic model. Int. Journal of Fracture, in revision.



\bibitem{CMPer-Lipton3}
Lipton, R., 2014. Dynamic brittle fracture as a small horizon limit of
  peridynamics. Journal of Elasticity 117~(1), 21--50.

\bibitem{CMPer-Lipton}
Lipton, R., 2016. Cohesive dynamics and brittle fracture. Journal of Elasticity
  124~(2), 143--191.
  
\bibitem{CMPer-Lipton4}
Lipton, R., Said, E., and Jha, P.~K., 2018. Dynamic brittle fracture from nonlocal
double-well potentials: A state-based model. Handbook of Nonlocal Continuum
Mechanics for Materials and Structures, 1--27. \urlprefix\url{https://doi.org/10.1007/978-3-319-22977-5_33-1}

\bibitem{CMPer-Lipton5}
Lipton, R., Said, E., and Jha, P.~K., 2018. Free damage propagation with memory. Journal of Elasticity, 133~(2), 129--153.




\bibitem{LM}
Lions, J.~L. and Magenes, E., 1972. Nonhomogeneous Boundary Value Problems and Applications. Vol. 1,
Springer Verlag, Berlin.

\bibitem{McLean}
McLean, W., 2000. Strongly Elliptic Systems and Boundary Integral Equations. 
Cambridge University Press. Cambridge, U.K.

\bibitem{Miehe} 
Miehe, C., Hofacker, M., and Welschinger, F.  2010. {A phase field model for rate-independent crack propagation: Robust algorithmic implementation based on operator splits.} Computer Methods in Applied Mechanics and Engineering, {199} 2765--2778.


\bibitem{CMPer-Mengesha2}
Mengesha, T. and Du, Q., 2015. On the variational limit of a class of nonlocal
  functionals related to peridynamics. Nonlinearity 28~(11), 3999.
  


\bibitem{Nicaise}
Nicaise, S., Sandig, A.-M., 2007.  Dynamic crack propagation in a 2D elastic body: the out-of-plane case. J. Math. 
Anal. Appl. 329, 1--30.

\bibitem{RaviChandar}
Ravi-Chandar, K., 2004. Dynamic Fracture. Elsevier.  Oxford, UK.

\bibitem{Schwab}
Schwab, CH. , 1998. p- and hp- Finite Element Methods. Clarendon Press. Oxford, UK.

\bibitem{Ortiz}
Schmidt, B., Fraternali, F., and Ortiz, M. 2009. {Eigenfracture: an eigendeformation approach to variational fracture.} Multiscale Model. Simul. {7} 1237--1266.




\bibitem{CMPer-Silling}
Silling, S.~A., 2000. Reformulation of elasticity theory for discontinuities
  and long-range forces. Journal of the Mechanics and Physics of Solids 48~(1),
  175--209.


\bibitem{States}
Silling, S.~A., Epton, M., Weckner, O., Xu, J. and  Askari, E., 2007. Peridynamic
  states and constitutive modeling. Journal of Elasticity 88~(2), 151--184.



\bibitem{silling05}
Silling S.~A. and Askari, E., 2005.  A meshfree method based on the peridynamic model of solid mechanics. Computers and Structures {83},  1526--1535.

\bibitem{Slepian}
Slepian, Y., 2002. Models and Phenomena in Fracture Mechanics. Foundations of Engineering Mechanics. Springer-Verlag. Berlin.

\bibitem{ParksTrasketal}
Trask, N., You, H., Yu, Y., and Parks, M.~L., 2019. An asymptotically compatible mesh free quadrature rule for nonlocal problems with applications to peridynamics. Computer Methods in Applied Mechanics and Engineering 343, 151--165.


\bibitem{trusk}
Truskinovsky, L., 1996. Fracture as a phase transition. Contemporary Research in the  Mechanics and Mathematics of Materials. R.C. Batra and M.F. Beatty Editors. CIMME, Barcelona.


\end{thebibliography}
\end{document}